\documentclass[11pt]{article}%
\usepackage{graphics}
\usepackage{amsfonts}
\usepackage{amsmath}
\usepackage{amssymb}
\usepackage{graphicx}%
\newtheorem{theorem}{Theorem}
\newtheorem{assumption}[theorem]{Assumption}
\newtheorem{algorithm}[theorem]{Algorithm}
\newtheorem{corollary}[theorem]{Corollary}
\newtheorem{lemma}[theorem]{Lemma}
\newtheorem{proposition}[theorem]{Proposition}
\newtheorem{remark}[theorem]{Remark}
\newenvironment{proof}[1][Proof]{\noindent\textbf{#1.} }{\ \rule{0.5em}{0.5em}}

\providecommand{\U}[1]{\protect\rule{.1in}{.1in}}
\providecommand{\U}[1]{\protect\rule{.1in}{.1in}}
\providecommand{\U}[1]{\protect\rule{.1in}{.1in}}
\begin{document}

\title{Multispace and Multilevel BDDC}
\author{Jan Mandel\footnotemark[1] \footnotemark[2]
\and Bed\v{r}ich Soused\'{\i}k\footnotemark[2] \footnotemark[3]
\and Clark R. Dohrmann\footnotemark[4]}
\maketitle

\begin{abstract}
BDDC method is the most advanced method from the Balancing family of iterative
substructuring methods for the solution of large systems of linear algebraic
equations arising from discretization of elliptic boundary value problems. In
the case of many substructures, solving the coarse problem exactly becomes a
bottleneck. Since the coarse problem in BDDC has the same structure as the
original problem, it is straightforward to apply the BDDC method recursively
to solve the coarse problem only approximately. In this paper, we formulate a
new family of abstract Multispace BDDC methods and give condition number
bounds from the abstract additive Schwarz preconditioning theory. The
Multilevel BDDC is then treated as a special case of the Multispace BDDC and
abstract multilevel condition number bounds are given. The abstract bounds
yield polylogarithmic condition number bounds for an arbitrary fixed number of
levels and scalar elliptic problems discretized by finite elements in two and
three spatial dimensions. Numerical experiments confirm the theory.



\end{abstract}

\footnotetext[1]{Corresponding author. \texttt{jan.mandel@cudenver.edu}.}
\footnotetext[2]{Department of Mathematical Sciences, University of Colorado
Denver, Denver, CO 80217-3364, USA. Supported in part by the National Science
Foundation under grants \mbox{CNS-0325314}, \mbox{DMS-071387}, and
\mbox{CNS-0719641}.} \footnotetext[3]{Department of Mathematics, Faculty of
Civil Engineering, Czech Technical University, Th\'akurova 7, 166 29 Prague 6,
Czech Republic. Supported by the program of the Information society of the
Academy of Sciences of the Czech Republic \mbox{1ET400760509} and by the Grant
Agency of the Czech Republic \mbox{GA \v{C}R 106/08/0403}.} \footnotetext[4]%
{Structural Dynamics Research Department, Sandia National Laboratories,
Albuquerque NM 87185-0847, USA. Sandia is a multiprogram laboratory operated by Sandia
Corporation, a Lockheed Martin Company, for the United States Department
of Energy's National Nuclear Security Administration under Contract
DE-AC04-94-AL85000.}

\section{Introduction}

\label{sec:introduction}

The BDDC (Balancing Domain Decomposition by Constraints) method by
Dohrmann~\cite{Dohrmann-2003-PSC} is the most advanced method from the BDD
family introduced by Mandel~\cite{Mandel-1993-BDD}. It is a Neumann-Neumann
iterative substructuring method of Schwarz type~\cite{Dryja-1995-SMN} that
iterates on the system of primal variables reduced to the interfaces between
the substructures. The BDDC\ method is closely related to the FETI-DP method
(Finite Element Tearing and Interconnecting - Dual, Primal) by Farhat et
al.~\cite{Farhat-2001-FDP,Farhat-2000-SDP}. FETI-DP\ is a dual method that
iterates on a system for Lagrange multipliers that enforce continuity on the
interfaces, with some \textquotedblleft coarse\textquotedblright\ variables
treated as primal, and it is a further development of the FETI method by
Farhat and Roux~\cite{Farhat-1991-MFE}. Polylogarithmic condition number
estimates for BDDC were obtained in~\cite{Mandel-2003-CBD,Mandel-2005-ATP} and
a proof that the eigenvalues of BDDC and FETI-DP are actually the same except
for eigenvalue equal to one was given in Mandel et al.~\cite{Mandel-2005-ATP}.
Simpler proofs of the equality of eigenvalues were obtained by Li and
Widlund~\cite{Li-2006-FBB}, and also by Brenner and
Sung~\cite{Brenner-2007-BFW}, who also gave an example when BDDC\ has an
eigenvalue equal to one but FETI-DP\ does not. In the case of many
substructures, solving the coarse problem exactly becomes a bottleneck.
However, since the coarse problem in BDDC has the same form as the original
problem, the BDDC method can be applied recursively to solve the coarse
problem only approximately. This leads to a multilevel form of BDDC in a
straightforward manner, see Dohrmann~\cite{Dohrmann-2003-PSC}. Polylogarithmic
condition number bounds for three-level BDDC (BDDC with two coarse levels)
were proved in two and three spatial dimensions by
Tu~\cite{Tu-2007-TBT,Tu-2007-TBT3D}.

In this paper, we present a new abstract Multispace BDDC\ method. The method
extends a simple variational setting of BDDC from Mandel and
Soused\'{\i}k~\cite{Mandel-2007-ASF}, which could be understood as an abstract
version of BDDC\ by partial subassembly in Li and Widlund~\cite{Li-2007-UIS}.
However, we do not adopt their change of variables, which does not seem to be
suitable in our abstract setting. We provide a condition number estimate for
the abstract Multispace BDDC method, which generalizes the estimate for a
single space from~\cite{Mandel-2007-ASF}. The proof is based on the abstract
additive Schwarz theory by Dryja and Widlund~\cite{Dryja-1995-SMN}. Many BDDC
formulations (with an explicit treatment of substructure interiors, after
reduction to substructure interfaces, with two levels, and multilevel) are
then viewed as abstract Multispace BDDC with a suitable choice of spaces and
operators, and abstract condition number estimates for those BDDC methods
follow. This result in turn gives a polylogarithmic condition number bound for
Multilevel BDDC applied to a second-order scalar elliptic model problems, with
an arbitrary number of levels. A brief presentation of the main results of the
paper, without proofs and with a simplified formulation of the Multispace BDDC
estimate, is contained in the conference paper~\cite{Mandel-2007-OMB}.

The paper is organized as follows. In Sec.~\ref{sec:problem-setting} we
introduce the abstract problem setting. In Sec.~\ref{sec:abstract-multispace}
we formulate an abstract Multispace BDDC as an additive Schwarz
preconditioner. In Sec.~\ref{sec:FE-setting} we introduce the settings of a
model problem using finite element discretization. In Sec.~\ref{sec:one-level}
we recall the algorithm of the original (two-level) BDDC method and formulate
it as a Multispace BDDC. In Sec.~\ref{sec:multilevel-bddc} we generalize the
algorithm to obtain Multilevel BDDC and we also give an abstract condition
number bound. In Sec.~\ref{sec:multilevel-condition} we derive the condition
number bound for the model problem. Finally, in
Sec.~\ref{sec:numerical-examples}, we report on numerical results.

\section{Abstract Problem Setting}

\label{sec:problem-setting}

We wish to solve an abstract linear problem
\begin{equation}
u\in X:a(u,v)=\left\langle f_{X},v\right\rangle ,\quad\forall v\in X,
\label{eq:problem}%
\end{equation}
where $X$\ is a finite dimensional linear space, $a\left(  \cdot,\cdot\right)
$ is a symmetric positive definite bilinear form defined on $X$, $f_{X}\in
X^{\prime}$ is the right-hand side with $X^{\prime}$ denoting the dual space
of $X$, and $\left\langle \cdot,\cdot\right\rangle $ is the duality pairing.
The form $a\left(  \cdot,\cdot\right)  $ is also called the energy inner
product, the value of the quadratic form $a\left(  u,u\right)  $ is called the
energy of $u$, and the norm $\left\Vert u\right\Vert _{a}=a\left(  u,u\right)
^{1/2}$ is called the energy norm. The operator $A_{X}:X\mapsto X^{\prime}$
associated with $a$ is defined by
\[
a(u,v)=\left\langle A_{X}u,v\right\rangle ,\quad\forall u,v\in X.
\]

A preconditioner is a mapping $B:X^{\prime}\rightarrow X$ and we will look for
preconditioners such that $\left\langle r,Br\right\rangle $ is also symmetric
and positive definite on $X^{\prime}$. It is well known that then
$BA_{X}:X\rightarrow X$ has only real positive eigenvalues, and convergence of
the preconditioned conjugate gradients method is bounded using the condition
number%
\[
\kappa=\frac{\lambda_{\max}(BA_{X})}{\lambda_{\min}(BA_{X})},
\]
which we wish to bound above.

All abstract spaces in this paper are finite dimensional linear spaces and we
make no distinction between a linear operator and its matrix.

\section{Abstract Multispace BDDC}

\label{sec:abstract-multispace}

To introduce abstract Multispace BDDC\ preconditioner, suppose that the
bilinear form $a$ is defined and symmetric positive semidefinite on some
larger space $W\supset X$. The preconditioner is derived from the abstract
additive Schwarz theory, however we decompose some space between $X$ and $W$
rather than $X$ as it would be done in the additive Schwarz method: In the
design of the preconditioner, we choose spaces $V_{k}$, $k=1,\ldots,M$, such
that
\begin{equation}
X\subset\sum_{k=1}^{M}V_{k}\subset W. \label{eq:multispace-BDDC-spaces}%
\end{equation}

\begin{assumption}
\label{assum:pos-def}The form $a\left(  \cdot,\cdot\right)  $ is positive
definite on each $V_{k}$ separately.
\end{assumption}

\begin{algorithm}
[Abstract Multispace BDDC]\label{alg:multispace-bddc} Given spaces $V_{k}$ and
linear operators $Q_{k}$, $k=1,\ldots,M$ such that $a\left(  \cdot
,\cdot\right)  $ is positive definite on each space $V_{k}$, and%
\[
X\subset\sum_{k=1}^{M}V_{k},\quad Q_{k}:V_{k}\rightarrow X,
\]
define the preconditioner $B:r\in X^{\prime}\longmapsto u\in X$ by%
\begin{equation}
B:r\mapsto u=\sum_{k=1}^{M}Q_{k}v_{k},\quad v_{k}\in V_{k}:\quad a\left(
v_{k},z_{k}\right)  =\left\langle r,Q_{k}z_{k}\right\rangle ,\quad\forall
z_{k}\in V_{k}. \label{def:abs-mult-BDDC}%
\end{equation}

\end{algorithm}

We formulate the condition number bound first in the full strength allowed by
the proof. The bound used in the rest of this paper will be a corollary.

\begin{theorem}
\label{thm:BDDC-M}Define for all $k=1,\ldots,M$ the spaces $V_{k}%
^{\mathcal{M}}$ by%
\[
V_{k}^{\mathcal{M}}=\left\{  v_{k}\in V_{k}:\forall z_{k}\in V_{k}:Q_{k}%
v_{k}=Q_{k}z_{k}\Longrightarrow\left\Vert v_{k}\right\Vert _{a}^{2}%
\leq\left\Vert z_{k}\right\Vert _{a}^{2}\right\}  .
\]
If there exist constants $C_{0},$ $\omega,$ and a~symmetric matrix
$\mathcal{E}=(e_{ij})_{i,j=1}^{M}$, such that%
\begin{align}
&  \forall u\in X\quad\exists v_{k}\in V_{k},\text{ }k=1,\ldots,M:u=\sum
_{k=1}^{M}Q_{k}v_{k},\text{ }\sum_{k=1}^{M}\left\Vert v_{k}\right\Vert
_{a}^{2}\leq C_{0}\left\Vert u\right\Vert _{a}^{2}\label{eq:upper-Wi}\\
&  \forall k=1,\ldots,M\quad\forall v_{k}\in V_{k}^{\mathcal{M}}:\left\Vert
Q_{k}v_{k}\right\Vert _{a}^{2}\leq\omega\left\Vert v_{k}\right\Vert _{a}%
^{2}\label{eq:lower-Wi}\\
&  \forall z_{k}\in Q_{k}V_{k},\text{ }k=1,\ldots,M:a\left(  z_{i}%
,z_{j}\right)  \leq e_{ij}\left\Vert z_{i}\right\Vert _{a}\left\Vert
z_{j}\right\Vert _{a},\quad\label{eq:cauchy-W}%
\end{align}
then the preconditioner from Algorithm~\ref{alg:multispace-bddc} satisfies
\[
\kappa=\frac{\lambda_{\max}(BA_{X})}{\lambda_{\min}(BA_{X})}\leq C_{0}%
\omega\rho(\mathcal{E}).
\]

\end{theorem}

\begin{proof}
We interpret the Multispace BDDC preconditioner as an abstract additive
Schwarz method. An abstract additive Schwarz method is specified by a
decomposition of the space $X$ into subspaces,%
\begin{equation}
X=X_{1}+...+X_{M}, \label{eq:decomposition-schwarz}%
\end{equation}
and by symmetric positive definite bilinear forms $b_{i}$ on $X_{i}$. The
preconditioner is a linear operator
\[
B:X^{\prime}\rightarrow X,\qquad B:r\mapsto u,
\]
defined by solving the following variational problems on the subspaces and
adding the results,%
\begin{equation}
B:r\mapsto u=\sum\limits_{k=1}^{M}u_{k},\quad u_{k}\in X_{k}:\quad b_{k}%
(u_{k},y_{k})=\left\langle r,y_{k}\right\rangle ,\quad\forall y_{k}\in X_{k}.
\label{eq:def-uM}%
\end{equation}

Dryja and Widlund \cite{Dryja-1995-SMN} proved that if there exist constants
$C_{0},$ $\omega,$ and a~symmetric matrix $\mathcal{E}=(e_{ij})_{i,j=1}^{M}$,
such that%
\begin{align}
\forall u  &  \in X\text{ }\exists u_{k}\in X_{k},\text{ }k=1,\ldots
,M:u=\sum_{k=1}^{M}u_{k},\text{ }\sum_{k=1}^{M}\left\Vert u_{k}\right\Vert
_{b_{k}}^{2}\leq C_{0}\left\Vert u\right\Vert _{a}^{2}\label{eq:upper-bi}\\
\forall k  &  =1,\ldots,M\text{ }\forall u_{k}\in X_{k}:\left\Vert
u_{k}\right\Vert _{a}^{2}\leq\omega\left\Vert u_{k}\right\Vert _{b_{k}}%
^{2}\label{eq:lower-bi}\\
\forall u_{k}  &  \in X_{k},\text{ }k=1,\ldots,M:a(u_{i},u_{j})\leq
e_{ij}\left\Vert u_{i}\right\Vert _{a}\left\Vert u_{j}\right\Vert _{a}
\label{eq:str-cauchy}%
\end{align}
then
\[
\kappa=\frac{\lambda_{\max}(BA_{X})}{\lambda_{\min}(BA_{X})}\leq C_{0}%
\omega\rho(\mathcal{E}),
\]
where $\rho$ is the spectral radius.

Now the idea of the proof is essentially to map the assumptions of the
abstract additive Schwarz estimate from the decomposition
(\ref{eq:decomposition-schwarz}) of the space $X$ to the decomposition
(\ref{eq:multispace-BDDC-spaces}). Define the spaces
\[
X_{k}=Q_{k}V_{k}.
\]
We will show that the preconditioner (\ref{def:abs-mult-BDDC}) satisfies
(\ref{eq:def-uM}), where $b_{k}$\ is defined by%
\begin{equation}
b_{k}(u_{k},y_{k})=a\left(  G_{k}x,G_{k}z\right)  ,\quad x,z\in X,\quad
u_{k}=Q_{k}G_{k}x,\quad y_{k}=Q_{k}G_{k}z. \label{eq:def-bk}%
\end{equation}
with the operators $G_{k}:X\rightarrow V_{k}^{\mathcal{M}}$ defined by%
\begin{equation}
G_{k}:u\mapsto v_{k},\quad\frac{1}{2}a\left(  v_{k},v_{k}\right)
\rightarrow\min,\text{ s.t. }v_{k}\in V_{k}^{\mathcal{M}},\text{ }u=\sum
_{k=1}^{M}Q_{k}v_{k}, \label{eq:op-G-minim}%
\end{equation}
First, from the definition of operators $G_{k}$, spaces $X_{k},$ and because
$a$ is positive definite on $V_{k}$ by Assumption (\ref{assum:pos-def}), it
follows that $G_{k}x$ and $G_{k}z$ in (\ref{eq:def-bk}) exist and are unique,
so $b_{k}$ is defined correctly. To prove (\ref{eq:def-uM}), let $v_{k}$\ be
as in (\ref{def:abs-mult-BDDC}) and note that $v_{k}$\ is the solution of
\[
\frac{1}{2}a\left(  v_{k},v_{k}\right)  -\left\langle r,Q_{k}v_{k}%
\right\rangle \rightarrow\min,\quad v_{k}\in V_{k}.
\]
Consequently, the preconditioner (\ref{def:abs-mult-BDDC}) is an abstract
additive Schwarz method and we only need to verify the inequalities
(\ref{eq:upper-bi})--(\ref{eq:str-cauchy}). To prove (\ref{eq:upper-bi}), let
$u\in X$. Then, with $v_{k}$ from the assumption (\ref{eq:upper-Wi}) and with
$u_{k}=Q_{k}G_{k}v_{k}$ as in (\ref{eq:def-bk}), it follows that%
\[
u=\sum_{k=1}^{M}u_{k},\quad\sum_{k=1}^{M}\left\Vert u_{k}\right\Vert _{b_{k}%
}^{2}=\sum_{k=1}^{M}\left\Vert v_{k}\right\Vert _{a}^{2}\leq C_{0}\left\Vert
u\right\Vert _{a}^{2}.
\]
Next, let $u_{k}\in X_{k}$. From the definitions of $X_{k}$ and $V_{k}%
^{\mathcal{M}}$, it follows that there exist unique $v_{k}\in V_{k}%
^{\mathcal{M}}$ such that $u_{k}=Q_{k}v_{k}$. Using the assumption
(\ref{eq:lower-Wi}) and the definition of $b_{k}$ in (\ref{eq:def-bk}), we
get
\[
\left\Vert u_{k}\right\Vert _{a}^{2}=\left\Vert Q_{k}v_{k}\right\Vert _{a}%
^{2}\leq\omega\left\Vert v_{k}\right\Vert _{a}^{2}=\omega\left\Vert
u_{k}\right\Vert _{b_{k}}^{2},
\]
which gives (\ref{eq:lower-bi}). Finally, (\ref{eq:cauchy-W}) is the same as
(\ref{eq:str-cauchy}).
\end{proof}

The next Corollary was given without proof in~\cite[Lemma~1]{Mandel-2007-OMB}.
This is the special case of Theorem~\ref{thm:BDDC-M} that will be actually
used. In the case when $M=1$, this result was proved in~\cite{Mandel-2007-ASF}.

\begin{corollary}
\label{cor:multispace-bddc}Assume that the subspaces $V_{k}$ are energy
orthogonal, the operators $Q_{k}$ are projections, $a\left(  \cdot
,\cdot\right)  $ is positive definite on each space $V_{k}$, and%
\begin{equation}
\forall u\in X:\left[  u=\sum_{k=1}^{M}v_{k},\ v_{k}\in V_{k}\right]
\Longrightarrow u=\sum_{k=1}^{M}Q_{k}v_{k}\text{.} \label{eq:dec-unity}%
\end{equation}
Then the abstract Multispace BDDC preconditioner from Algorithm
\ref{alg:multispace-bddc} satisfies%
\begin{equation}
{\kappa=\frac{\lambda_{\max}(BA_{X})}{\lambda_{\min}(BA_{X})}\leq\omega
=\max_{k}\sup_{v_{k}\in V_{k}}\frac{\left\Vert Q_{k}v_{k}\right\Vert _{a}^{2}%
}{\left\Vert v_{k}\right\Vert _{a}^{2}}}\text{ }{.}
\label{eq:multispace-bddc-bound}%
\end{equation}

\end{corollary}

\begin{proof}
We only need to verify the assumptions of Theorem~\ref{thm:BDDC-M}. Let $u\in
X$ and choose $v_{k}$ as the energy orthogonal projections of $u$ on $V_{k}$.
First, since the spaces $V_{k}$ are energy orthogonal, $u=%
{\textstyle\sum}
v_{k}$, $Q_{k}$\ are projections, and from~(\ref{eq:dec-unity}) $u=%
{\textstyle\sum}
Q_{k}v_{k}$, we get that $\left\Vert u\right\Vert _{a}^{2}=%
{\textstyle\sum}
\left\Vert v_{k}\right\Vert _{a}^{2}$ which proves (\ref{eq:upper-Wi})\ with
$C_{0}=1$. Next, the assumption (\ref{eq:lower-Wi}) becomes the definition of
${\omega} $\ in (\ref{eq:multispace-bddc-bound}). Finally, (\ref{eq:cauchy-W})
with $\mathcal{E}=I$ follows from the orthogonality of subspaces $V_{k}$.
\end{proof}

\begin{remark}
The assumption (\ref{eq:dec-unity}) can be written as%
\[
\left.  \sum_{k=1}^{M}Q_{k}P_{k}\right\vert _{X}=I,
\]
where $P_{k}$ is the $a$-orthogonal projection from $\bigoplus\nolimits_{j=1}%
^{M}V_{j}$ onto $V_{k}$. Hence, the property (\ref{eq:dec-unity}) is a type of
decomposition of unity.

In the case when $M=1$, (\ref{eq:dec-unity}) means that the projection $Q_{1}
$ is onto $X$.
\end{remark}

\section{Finite Element Problem Setting}

\label{sec:FE-setting}

Let $\Omega$ be a bounded domain in $\mathbb{R}^{d}$, $d=2$ or $3$, decomposed
into $N$ nonoverlapping subdomains $\Omega^{s}$, $s=1,...,N$, which form a
conforming triangulation of the domain $\Omega$. Subdomains will be also
called substructures. Each substructure is a union of Lagrangian $P1$ or $Q1$
finite elements with characteristic mesh size $h$, and the nodes of the finite
elements between substructures coincide. The nodes contained in the
intersection of at least two substructures are called boundary nodes. The
union of all boundary nodes is called the interface $\Gamma$. The
interface$~\Gamma$ is a union of three different types of open sets:
\emph{faces}, \emph{edges}, and \emph{vertices}. The substructure vertices
will be also called corners. For the case of regular substructures, such as
cubes or tetrahedrons, we can use standard geometric definition of faces,
edges, and vertices; cf., e.g.,~\cite{Klawonn-2006-DPF} for a more general definition.

In this paper, we find it more convenient to use the notation of abstract
linear spaces and linear operators between them instead of the space
$\mathbb{R}^{n}$ and matrices. The results can be easily converted to the
matrix language by choosing a finite element basis. The space of the finite
element functions on $\Omega$ will be denoted as $U$. Let $W^{s}$ be the space
of finite element functions\ on substructure $\Omega^{s}$, such that all of
their degrees of freedom on $\partial\Omega^{s}\cap\partial\Omega$ are zero.
Let%
\[
W=W^{1}\times\cdots\times W^{N},
\]
and consider a bilinear form arising from the second-order scalar elliptic
problem as%
\begin{equation}
a\left(  u,v\right)  =\sum_{s=1}^{N}\int_{\Omega^{S}}\nabla u\nabla
v\,dx,\quad u,v\in W. \label{eq:scalar-bilinear-form}%
\end{equation}

Now $U\subset W$ is the subspace of all functions from $W$ that are continuous
across the substructure interfaces. We are interested in the solution of the
problem~(\ref{eq:problem}) with $X=U$,%
\begin{equation}
u\in U:a(u,v)=\left\langle f,v\right\rangle ,\quad\forall v\in U,
\label{eq:problem-full-space}%
\end{equation}
where the bilinear form $a$ is associated on the space $U$\ with the system
operator$~A$, defined by
\begin{equation}
A:U\mapsto U^{\prime},\quad a(u,v)=\left\langle Au,v\right\rangle \text{ for
all }u,v\in U, \label{eq:def-A}%
\end{equation}
and $f\in U^{\prime}$ is the right-hand side. Hence,
(\ref{eq:problem-full-space}) is equivalent to%
\begin{equation}
Au=f. \label{eq:problem-full-space-algebraic}%
\end{equation}

Define $U_{I}\subset U$ as the subspace of functions that are zero on the
interface $\Gamma$,
i.e., the \textquotedblleft interior\textquotedblright\ functions. Denote by
$P$ the energy orthogonal projection from $W$ onto $U_{I}$,%
\[
P:w\in W\longmapsto v_{I}\in U_{I}:a\left(  v_{I},z_{I}\right)  =a\left(
w,z_{I}\right)  ,\quad\forall z_{I}\in U_{I}.
\]
Functions from $\left(  I-P\right)  W$, i.e., from the nullspace of $P,$ are
called discrete harmonic; these functions are $a$-orthogonal to $U_{I}$ and
energy minimal with respect to increments in $U_{I}$. Next, let $\widehat{W}$
be the space of all discrete harmonic functions that are continuous across
substructure boundaries, that is
\begin{equation}
\widehat{W}=\left(  I-P\right)  U. \label{eq:discrete-harm}%
\end{equation}
In particular,
\begin{equation}
U=U_{I}\oplus\widehat{W},\quad U_{I}\perp_{a}\widehat{W}.
\label{eq:int-harm-dec}%
\end{equation}

A common approach in substructuring is to reduce the problem to the interface.
The problem (\ref{eq:problem-full-space}) is equivalent to two independent
problems on the energy orthogonal subspaces $U_{I}$ and $\widehat{W}$, and the
solution $u$ satisfies $u=u_{I}+\widehat{u}$, where%

\begin{align}
u  &  \in U_{I}:a(u_{I},v_{I})=\left\langle f,v_{I}\right\rangle ,\quad\forall
v_{I}\in U_{I},\label{eq:problem-int-1}\\
u  &  \in\widehat{W}:a(\widehat{u},\widehat{v})=\left\langle f,\widehat
{v}\right\rangle ,\quad\forall\widehat{v}\in\widehat{W}.
\label{eq:problem-reduced}%
\end{align}
The solution of the interior problem (\ref{eq:problem-int-1}) decomposes into
independent problems, one per each substructure. The reduced problem
(\ref{eq:problem-reduced}) is then solved by preconditioned conjugate
gradients. The reduced problem (\ref{eq:problem-reduced}) is usually written
equivalently as%
\[
u\in\widehat{W}:s(\widehat{u},\widehat{v})=\left\langle g,\widehat
{v}\right\rangle ,\quad\forall\widehat{v}\in\widehat{W},
\]
where $s$ is the form $a$ restricted on the subspace $\widehat{W}$, and $g$ is
the reduced right hand side, i.e., the functional $f$ restricted to the space
$\widehat{W}$. The reduced right-hand side $g$ is usually written as%
\begin{equation}
\left\langle g,\widehat{v}\right\rangle =\left\langle f,\widehat
{v}\right\rangle -a(u_{I},\widehat{v}),\quad\forall\widehat{v}\in\widehat{W},
\label{eq:reduced-rhs}%
\end{equation}
because $a(u_{I},\widehat{v})=0$ by (\ref{eq:int-harm-dec}). In the
implementation, the process of passing to the reduced problem becomes the
elimination of the internal degrees of freedom of the substructures, also
known as static condensation. The matrix of the reduced bilinear form $s$ in
the basis defined by interface degrees of freedom becomes the Schur
complement, and (\ref{eq:reduced-rhs}) becomes the reduced right-hand side.
For details on the matrix formulation, see, e.g., \cite[Sec. 4.6]%
{Smith-1996-DD} or \cite[Sec. 4.3]{Toselli-2005-DDM}.

The BDDC\ method is a two-level preconditioner characterized by the selection
of certain \emph{coarse degrees of freedom}, such as values at the corners and
averages over edges or faces of substructures. Define $\widetilde{W}\subset W$
as the subspace of all functions such that the values of any coarse degrees of
freedom have a common value for all relevant substructures and vanish on
$\partial\Omega,$ and $\widetilde{W}_{\Delta}\subset W$ as the subspace of all
function such that their coarse degrees of freedom vanish. Next, define
$\widetilde{W}_{\Pi}$ as the subspace of all functions such that their coarse
degrees of freedom between adjacent substructures coincide, and such that
their energy is minimal. Clearly, functions in $\widetilde{W}_{\Pi}$ are
uniquely determined by the values of their coarse degrees of freedom, and
\begin{equation}
\widetilde{W}_{\Delta}\perp_{a}\widetilde{W}_{\Pi},\text{\quad and\quad
}\widetilde{W}=\widetilde{W}_{\Delta}\oplus\widetilde{W}_{\Pi}.
\label{eq:tilde-dec}%
\end{equation}

We assume that%
\begin{equation}
a\text{ is positive definite on }\widetilde{W}. \label{eq:pos-def}%
\end{equation}
That is the case when $a$ is positive definite on the space $U$, where the
problem~(\ref{eq:problem}) is posed, and there are sufficiently many coarse
degrees of freedom. We further assume that the coarse degrees of freedom are
zero on all functions from $U_{I}$, that is,%
\begin{equation}
U_{I}\subset\widetilde{W}_{\Delta}. \label{eq:coarse-int}%
\end{equation}
In other words, the coarse degrees of freedom depend on the values on
substructure boundaries only. From (\ref{eq:tilde-dec}) and
(\ref{eq:coarse-int}), it follows that the functions in $\widetilde{W}_{\Pi}$
are discrete harmonic, that is,%
\begin{equation}
\widetilde{W}_{\Pi}=\left(  I-P\right)  \widetilde{W}_{\Pi}.
\label{eq:coarse-is-discrete-harmonic}%
\end{equation}

Next, let $E$ be a projection from $\widetilde{W}$ onto $U$, defined by taking
some weighted average on substructure interfaces. That is, we assume that%
\begin{equation}
E:\widetilde{W}\rightarrow U,\quad EU=U,\quad E^{2}=E. \label{eq:E-onto-U}%
\end{equation}
Since a projection is the identity on its range, it follows that $E$ does not
change the interior degrees of freedom,
\begin{equation}
EU_{I}=U_{I}, \label{eq:int-unchanged}%
\end{equation}
since $U_{I}\subset U$. Finally, we show that the operator $\left(
I-P\right)  E$ is a projection. From (\ref{eq:int-unchanged}) it follows that
$E$ does not change interior degrees of freedom, so $EP=P$. Then, using the
fact that $I-P$ and $E$ are projections, we get%
\begin{align}
\left[  \left(  I-P\right)  E\right]  ^{2}  &  =\left(  I-P\right)  E\left(
I-P\right)  E\nonumber\\
&  =\left(  I-P\right)  \left(  E-P\right)  E\label{eq:(I-P)E}\\
&  =\left(  I-P\right)  \left(  I-P\right)  E=\left(  I-P\right)  E.\nonumber
\end{align}

\begin{remark}
In~\cite{Mandel-2005-ATP,Mandel-2007-ASF}, the whole analysis was done in
spaces of discrete harmonic functions after eliminating $U_{I}$, and the space
$\widehat{W}$ was the solution space. In particular, $\widetilde{W}$ consisted
of discrete harmonic functions only, while the same space here would be
$(I-P)\widetilde{W}$. The decomposition of this space used
in~\cite{Mandel-2005-ATP,Mandel-2007-ASF} would be\ in our context written as
\begin{equation}
(I-P)\widetilde{W}=(I-P)\widetilde{W}_{\Delta}\oplus\widetilde{W}_{\Pi}%
,\quad(I-P)\widetilde{W}_{\Delta}\perp_{a}\widetilde{W}_{\Pi}.
\label{eq:dharm-dec}%
\end{equation}

\end{remark}

In the next section, the space $X$ will be either $U$ or $\widehat{W}$.

\section{Two-level BDDC as Multispace BDDC}

\label{sec:one-level}

We show several different ways the original, two-level, BDDC\ algorithm can be
interpreted as multispace BDDC. We consider first BDDC applied to the reduced
problem~(\ref{eq:problem-reduced}), that is, (\ref{eq:problem}) with
$X=\widehat{W}$. This was the formulation considered in \cite{Mandel-2005-ATP}%
. Define the space of discrete harmonic functions with coarse degrees of
freedom continuous across the interface%
\[
\widetilde{W}_{\Gamma}=\left(  I-P\right)  \widetilde{W}.
\]
Because we work in the space of discrete harmonic functions and the output of
the averaging operator $E$ is not discrete harmonic, denote%
\begin{equation}
E_{\Gamma}=\left(  I-P\right)  E. \label{eq:E-gamma}%
\end{equation}
In an implementation, discrete harmonic functions are represented by the
values of their degrees of freedom on substructure interfaces, cf.,
e.g.~\cite{Toselli-2005-DDM}; hence, the definition (\ref{eq:E-gamma}) serves
formal purposes only, so that everything can be written in terms of discrete
harmonic functions without passing to the matrix formulation.

\begin{algorithm}
[\cite{Mandel-2007-ASF}, BDDC on the reduced problem]\label{alg:bddc-elim-int}%
Define the preconditioner $r\in\widehat{W}^{\prime}\longmapsto u\in\widehat
{W}$ by%
\begin{equation}
u=E_{\Gamma}w_{\Gamma},\quad w_{\Gamma}\in\widetilde{W}_{\Gamma}:a\left(
w_{\Gamma},z_{\Gamma}\right)  =\left\langle r,E_{\Gamma}z_{\Gamma
}\right\rangle ,\quad\forall z_{\Gamma}\in\widetilde{W}_{\Gamma}.
\label{eq:bddc-elim-int-alg-1}%
\end{equation}

\end{algorithm}

\begin{proposition}
[\cite{Mandel-2007-ASF}]The BDDC preconditioner on the reduced problem in
Algorithm~\ref{alg:bddc-elim-int} is the abstract Multispace BDDC from
Algorithm~\ref{alg:multispace-bddc} with $M=1$ and the space and operator
given by
\begin{equation}
X=\widehat{W},\quad V_{1}=\widetilde{W}_{\Gamma},\quad Q_{1}=E_{\Gamma}.
\label{eq:bddc-elim-int-1}%
\end{equation}
Also, the assumptions of Corollary~\ref{cor:multispace-bddc} are satisfied.
\end{proposition}

\begin{proof}
We only need to note that the bilinear form $a(\cdot,\cdot)$ is positive
definite on $\widetilde{W}_{\Gamma}\subset\widetilde{W}$ by~(\ref{eq:pos-def}%
), and the operator $E_{\Gamma}$ defined by~(\ref{eq:E-gamma}) is a projection
by~(\ref{eq:(I-P)E}). The projection $E_{\Gamma}$ is onto $\widehat{W}$
because $E$ is onto $U$ by (\ref{eq:E-onto-U}), and $I-P$ maps $U$ onto
$\widehat{W}$ by the definition of $\widehat{W}$ in (\ref{eq:discrete-harm}).
\end{proof}

Using the decomposition (\ref{eq:dharm-dec}), we can split the solution in the
space $\widetilde{W}_{\Gamma}$ into the independent solution of two
subproblems: mutually independent problems on substructures as the solution in
the space $\widetilde{W}_{\Gamma\Delta}=(I-P)\widetilde{W}_{\Delta}$,\ and a
solution of global coarse problem in the space $\widetilde{W}_{\Pi}$. The
space $\widetilde{W}_{\Gamma}$ has a decomposition
\begin{equation}
\widetilde{W}_{\Gamma}=\widetilde{W}_{\Gamma\Delta}\oplus\widetilde{W}_{\Pi
},\text{\quad and\quad}\widetilde{W}_{\Gamma\Delta}\perp_{a}\widetilde{W}%
_{\Pi}, \label{eq:dharm-dec-interface}%
\end{equation}
the same as the decomposition (\ref{eq:dharm-dec}), and
Algorithm~\ref{alg:bddc-elim-int} can be rewritten as follows.

\begin{algorithm}
[\cite{Mandel-2005-ATP}, BDDC on the reduced problem]%
\label{alg:bddc-elim-int-2}Define the preconditioner $r\in\widehat{W}^{\prime
}\longmapsto u\in\widehat{W}$ by $u=E_{\Gamma}\left(  w_{\Gamma\Delta}+w_{\Pi
}\right)  $, where%
\begin{align}
w_{\Gamma\Delta}  &  \in\widetilde{W}_{\Gamma\Delta}:a\left(  w_{\Gamma\Delta
},z_{\Gamma\Delta}\right)  =\left\langle r,E_{\Gamma}z_{\Gamma\Delta
}\right\rangle ,\quad\forall z_{\Gamma\Delta}\in\widetilde{W}_{\Gamma\Delta
},\label{eq:discr-harm-subs-corr}\\
w_{\Pi}  &  \in\widetilde{W}_{\Pi}:a\left(  w_{\Pi},z_{\Gamma\Pi}\right)
=\left\langle r,E_{\Gamma}z_{\Gamma\Pi}\right\rangle ,\quad\forall
z_{\Gamma\Pi}\in\widetilde{W}_{\Pi}. \label{eq:discr-harm-coarse-corr}%
\end{align}

\end{algorithm}

\begin{proposition}
The BDDC preconditioner on the reduced problem in
Algorithm~\ref{alg:bddc-elim-int-2} is the abstract Multispace BDDC from
Algorithm~\ref{alg:multispace-bddc} with $M=2$ and the spaces and operators
given by
\begin{equation}
X=\widehat{W},\quad V_{1}=\widetilde{W}_{\Gamma\Delta},\quad V_{2}%
=\widetilde{W}_{\Pi},\quad Q_{1}=Q_{2}=E_{\Gamma}. \label{eq:bddc-elim-int-2}%
\end{equation}
Also, the assumptions of Corollary~\ref{cor:multispace-bddc} are satisfied.
\end{proposition}

\begin{proof}
Let $r\in\widehat{W}^{\prime}$. Define the vectors $v_{i}$, $i=1,2$\ in
Multispace BDDC\ by (\ref{def:abs-mult-BDDC}) with $V_{i}$ and $Q_{i}$ given
by (\ref{eq:bddc-elim-int-2}). Let $u,$ $w_{\Gamma\Delta},$ $w_{\Pi}$ be the
quantities in Algorithm~\ref{alg:bddc-elim-int-2}, defined by
(\ref{eq:discr-harm-subs-corr})-(\ref{eq:discr-harm-coarse-corr}). Using the
decomposition (\ref{eq:dharm-dec-interface}), any $w_{\Gamma}\in\widetilde
{W}_{\Gamma}$ can be written uniquely as $w_{\Gamma}=w_{\Gamma\Delta}$%
+$w_{\Pi}$ for some $w_{\Gamma\Delta}$\ and $w_{\Pi}$ corresponding to
(\ref{def:abs-mult-BDDC})\ as $v_{1}=w_{\Gamma\Delta}$\ and $v_{2}=w_{\Pi}$,
and $u=E_{\Gamma}\left(  w_{\Gamma\Delta}+w_{\Pi}\right)  $.

To verify the assumptions of Corollary~\ref{cor:multispace-bddc}, note that
the decomposition (\ref{eq:dharm-dec-interface}) is $a-$orthogonal,
$a(\cdot,\cdot)$\ is positive definite on both $\widetilde{W}_{\Gamma\Delta}%
$\ and $\widetilde{W}_{\Pi}$ as subspaces of $\widetilde{W}_{\Gamma}$
by~(\ref{eq:pos-def}), and $E_{\Gamma}$ is a projection by (\ref{eq:(I-P)E}).
\end{proof}

Next, we present a BDDC formulation\ on the space $U$ with explicit treatment
of interior functions in the space $U_{I}$ as in
\cite{Dohrmann-2003-PSC,Mandel-2003-CBD}, i.e., in the way the BDDC algorithm
was originally formulated.

\begin{algorithm}
[\cite{Dohrmann-2003-PSC,Mandel-2003-CBD}, original BDDC]%
\label{alg:original-bddc} Define the preconditioner $r\in U^{\prime
}\longmapsto u\in U$\ as follows. Compute the interior pre-correction%
\begin{equation}
u_{I}\in U_{I}:a\left(  u_{I},z_{I}\right)  =\left\langle r,z_{I}\right\rangle
,\quad\forall z_{I}\in U_{I}. \label{eq:int-corr}%
\end{equation}
Set up the updated residual%
\begin{equation}
r_{B}\in U^{\prime},\quad\left\langle r_{B},v\right\rangle =\left\langle
r,v\right\rangle -a\left(  u_{I},v\right)  ,\quad\forall v\in U.
\label{eq:def-rb}%
\end{equation}
Compute the substructure correction
\begin{equation}
u_{\Delta}=Ew_{\Delta},\quad w_{\Delta}\in\widetilde{W}_{\Delta}:a\left(
w_{\Delta},z_{\Delta}\right)  =\left\langle r_{B},Ez_{\Delta}\right\rangle
,\quad\forall z_{\Delta}\in\widetilde{W}_{\Delta}. \label{eq:subs-corr}%
\end{equation}
Compute the coarse correction%
\begin{equation}
u_{\Pi}=Ew_{\Pi},\quad w_{\Pi}\in\widetilde{W}_{\Pi}:a\left(  w_{\Pi},z_{\Pi
}\right)  =\left\langle r_{B},Ez_{\Pi}\right\rangle ,\quad\forall z_{\Pi}%
\in\widetilde{W}_{\Pi}. \label{eq:coarse}%
\end{equation}
Add the corrections%
\[
u_{B}=u_{\Delta}+u_{\Pi}.
\]
Compute the interior post-correction%
\begin{equation}
v_{I}\in U_{I}:a\left(  v_{I},z_{I}\right)  =a\left(  u_{B},z_{I}\right)
,\quad\forall z_{I}\in U_{I}. \label{eq:int-post-corr}%
\end{equation}
Apply the combined corrections%
\begin{equation}
u=u_{B}-v_{I}+u_{I}. \label{eq:sol}%
\end{equation}

\end{algorithm}

The interior corrections (\ref{eq:int-corr}) and (\ref{eq:int-post-corr})
decompose into independent Dirichlet problems, one for each substructure. The
substructure correction (\ref{eq:subs-corr}) decomposes into independent
constrained Neumann problems, one for each substructure. Thus, the evaluation
of the preconditioner requires three problems to be solved in each
substructure, plus solution of the coarse problem~(\ref{eq:coarse}). In
addition, the substructure corrections can be solved in parallel with the
coarse problem.

\begin{remark}
As it is well known \cite{Dohrmann-2003-PSC}, the first interior correction
(\ref{eq:int-corr}) can be omitted in the implementation by starting the
iterations from an initial solution such that the residual in the interior of
the substructures is zero,
\[
a\left(  u,z_{I}\right)  -\left\langle f_{X},z_{I}\right\rangle =0,\quad
\forall z_{I}\in U_{I},
\]
i.e., such that the error is discrete harmonic. Then the output of the
preconditioner is discrete harmonic and thus the errors in all the CG
iterations (which are linear combinations of the original error and outputs
from the preconditioner) are also discrete harmonic by induction.
\end{remark}

The following proposition will be the starting point for the multilevel case.

\begin{proposition}
\label{prop:original-bddc-as-multispace}The original BDDC preconditioner in
Algorithm~\ref{alg:original-bddc} is the abstract Multispace BDDC from
Algorithm~\ref{alg:multispace-bddc} with $M=3$ and the spaces and operators
given by%
\begin{align}
X  &  =U,\quad V_{1}=U_{I},\quad V_{2}=(I-P)\widetilde{W}_{\Delta},\quad
V_{3}=\widetilde{W}_{\Pi},\label{eq:3-space}\\
Q_{1}  &  =I,\quad Q_{2}=Q_{3}=\left(  I-P\right)  E, \label{eq:3-proj}%
\end{align}
and the assumptions of Corollary~\ref{cor:multispace-bddc} are satisfied.
\end{proposition}

\begin{proof}
Let $r\in U^{\prime}$. Define the vectors $v_{i}$, $i=1,2,3$, in Multispace
BDDC by (\ref{def:abs-mult-BDDC}) with the spaces $V_{i}$ given by
(\ref{eq:3-space}) and with the operators $Q_{i}$ given by~(\ref{eq:3-proj}).
Let $u_{I}$, $r_{B}$, $w_{\Delta}$, $w_{\Pi}$, $u_{B}$, $v_{I}$, and $u$ be
the quantities in Algorithm~\ref{alg:original-bddc}, defined by
(\ref{eq:int-corr})-(\ref{eq:sol}).

First, with $V_{1}=U_{I}$, the definition of $v_{1}$ in
(\ref{def:abs-mult-BDDC}) with $k=1$ is identical to the definition of $u_{I}$
in (\ref{eq:int-corr}), so $u_{I}=v_{1}$.

Next, consider $w_{\Delta}\in\widetilde{W}_{\Delta}$ defined in
(\ref{eq:subs-corr}). We show that $w_{\Delta}$ satisfies
(\ref{def:abs-mult-BDDC}) with $k=2$, i.e., $v_{2}=w_{\Delta}$. So, let
$z_{\Delta}\in\widetilde{W}_{\Delta}$ be arbitrary. From (\ref{eq:subs-corr})
and (\ref{eq:def-rb}),
\begin{equation}
a\left(  w_{\Delta},z_{\Delta}\right)  =\left\langle r_{B},Ez_{\Delta
}\right\rangle =\left\langle r,Ez_{\Delta}\right\rangle -a\left(
u_{I},Ez_{\Delta}\right)  . \label{eq:w-delta}%
\end{equation}
Now from the definition of $u_{I}$ by (\ref{eq:int-corr}) and the fact that
$PEz_{\Delta}\in U_{I}$, we get
\begin{equation}
\left\langle r,PEz_{\Delta}\right\rangle -a\left(  u_{I},PEz_{\Delta}\right)
=0, \label{eq:pez}%
\end{equation}
and subtracting (\ref{eq:pez}) from (\ref{eq:w-delta}) gives%
\begin{align*}
a\left(  w_{\Delta},z_{\Delta}\right)   &  =\left\langle r,\left(  I-P\right)
Ez_{\Delta}\right\rangle -a\left(  u_{I},\left(  I-P\right)  Ez_{\Delta
}\right) \\
&  =\left\langle r,\left(  I-P\right)  Ez_{\Delta}\right\rangle ,
\end{align*}
because $a\left(  u_{I},\left(  I-P\right)  Ez_{\Delta}\right)  =0$ by
orthogonality. To verify (\ref{def:abs-mult-BDDC}), it is enough to show that
$Pw_{\Delta}=0;$ then $w_{\Delta}\in(I-P)\widetilde{W}_{\Delta}=V_{2}$. Since
$P$ is an $a$-orthogonal projection, it holds that%
\begin{equation}
a\left(  Pw_{\Delta},Pw_{\Delta}\right)  =a\left(  w_{\Delta},Pw_{\Delta
}\right)  =\left\langle r_{B},EPw_{\Delta}\right\rangle =0, \label{eq:APP}%
\end{equation}
where we have used $EU_{I}\subset U_{I}$ following the assumption
(\ref{eq:int-unchanged}) and the equality%
\[
\left\langle r_{B},z_{I}\right\rangle =\left\langle r,z_{I}\right\rangle
-a\left(  u_{I},z_{I}\right)  =0
\]
for any $z_{I}\in U_{I}$, which follows from (\ref{eq:def-rb}) and
(\ref{eq:int-corr}). Since $a$ is positive definite on $\widetilde{W}\supset
U_{I} $ by assumption (\ref{eq:pos-def}), it follows from (\ref{eq:APP})\ that
$Pw_{\Delta}=0$.

In exactly the same way, from (\ref{eq:coarse}) -- (\ref{eq:sol}), we get that
if $w_{\Pi}\in$ $\widetilde{W}_{\Pi}$ is defined by (\ref{eq:coarse}), then
$v_{3}=w_{\Pi}$ satisfies (\ref{def:abs-mult-BDDC}) with $k=3$. (The proof
that $Pw_{\Pi}=0$ can be simplified but there is nothing wrong with proceeding
exactly as for $w_{\Delta}$.)

Finally, from (\ref{eq:int-post-corr}), $v_{I}=P\left(  Ew_{\Delta}+Ew_{\Pi
}\right)  $, so%
\begin{align*}
u  &  =u_{I}+\left(  u_{B}-v_{I}\right) \\
&  =u_{I}+\left(  I-P\right)  Ew_{\Delta}+\left(  I-P\right)  Ew_{\Pi}\\
&  =Q_{1}v_{1}+Q_{2}v_{2}+Q_{3}v_{3}.
\end{align*}

It remains to verify the assumptions of Corollary~\ref{cor:multispace-bddc}.

First, the spaces $\widetilde{W}_{\Pi}$ and $\widetilde{W}_{\Delta}$ are $a
$-orthogonal by (\ref{eq:tilde-dec}) and, from (\ref{eq:coarse-int}),
\[
\left(  I-P\right)  \widetilde{W}_{\Delta}\subset\widetilde{W}_{\Delta},
\]
thus $\left(  I-P\right)  \widetilde{W}_{\Delta}\perp_{a}\widetilde{W}_{\Pi}
$. Clearly, $\left(  I-P\right)  \widetilde{W}_{\Delta}\perp_{a}U_{I}$. Since
$\widetilde{W}_{\Pi}$ consists of discrete harmonic functions from
(\ref{eq:coarse-is-discrete-harmonic}), so $\widetilde{W}_{\Pi}\perp_{a}U_{I}%
$, it follows that the spaces $V_{i}$, $i=1,2,3$, given by (\ref{eq:3-space}),
are $a$-orthogonal.


Next, $\left(  I-P\right)  E$ is by (\ref{eq:(I-P)E}) a projection, and so are
the operators $Q_{i}$ from~(\ref{eq:3-proj}).

It remains to prove the decomposition of unity (\ref{eq:dec-unity}). Let%
\begin{equation}
u^{\prime}=u_{I}+w_{\Delta}+w_{\Pi}\in U,\quad u_{I}\in U_{I},\quad w_{\Delta
}\in\left(  I-P\right)  \widetilde{W}_{\Delta},\quad w_{\Pi}\in\widetilde
{W}_{\Pi}, \label{eq:dec-u}%
\end{equation}
and let%
\[
v=u_{I}+\left(  I-P\right)  Ew_{\Delta}+\left(  I-P\right)  Ew_{\Pi}.
\]
From (\ref{eq:dec-u}), $w_{\Delta}+w_{\Pi}\in U$ since $u^{\prime}\in U$ and
$u_{I}\in U_{I}\subset U$. Then $E\left(  w_{\Delta}+w_{\Pi}\right)
=w_{\Delta}+w_{\Pi}$ by (\ref{eq:E-onto-U}), so%
\begin{align*}
v  &  =u_{I}+\left(  I-P\right)  E\left(  w_{\Delta}+w_{\Pi}\right) \\
&  =u_{I}+\left(  I-P\right)  \left(  w_{\Delta}+w_{\Pi}\right) \\
&  =u_{I}+w_{\Delta}+w_{\Pi}=u^{\prime},
\end{align*}
because both $w_{\Delta}$ and $w_{\Pi}$ are discrete harmonic.
\end{proof}

The next Theorem shows an equivalence of the three Algorithms introduced above.

\begin{theorem}
\label{thm:equiv}The eigenvalues of the preconditioned operators from
Algorithm~\ref{alg:bddc-elim-int}, and Algorithm~\ref{alg:bddc-elim-int-2} are
exactly the same. They are also the same as the eigenvalues from
Algorithm~\ref{alg:original-bddc}, except possibly for multiplicity of
eigenvalue equal to one.
\end{theorem}

\begin{proof}
From the decomposition (\ref{eq:dharm-dec-interface}), we can write any
$w\in\widetilde{W}_{\Gamma}$\ uniquely as $w=w_{\Delta}+w_{\Pi}$ for some
$w_{\Delta}\in\widetilde{W}_{\Gamma\Delta}$ and $w_{\Pi}\in\widetilde{W}_{\Pi
}$, so the preconditioned operators from Algorithms \ref{alg:bddc-elim-int}
and \ref{alg:bddc-elim-int-2}\ are spectrally equivalent and we need only to
show their spectral equivalence to the preconditioned operator from
Algorithm~\ref{alg:original-bddc}. First, we note that the operator
$A:U\mapsto U^{\prime}$ defined by~(\ref{eq:def-A}), and given in the block
form as%
\[
A=\left[
\begin{array}
[c]{cc}%
\mathcal{A}_{II} & \mathcal{A}_{I\Gamma}\\
\mathcal{A}_{\Gamma I} & \mathcal{A}_{\Gamma\Gamma}%
\end{array}
\right]  ,
\]
with blocks
\begin{align*}
\mathcal{A}_{II}  &  :U_{I}\rightarrow U_{I}^{\prime},\quad\mathcal{A}%
_{I\Gamma}:U_{I}\rightarrow\widehat{W}^{\prime},\\
\mathcal{A}_{\Gamma I}  &  :\widehat{W}\rightarrow U_{I}^{\prime}%
,\quad\mathcal{A}_{\Gamma\Gamma}:\widehat{W}\rightarrow\widehat{W}^{\prime},
\end{align*}
is block diagonal and $\mathcal{A}_{\Gamma I}=\mathcal{A}_{I\Gamma}=0$ for any
$u\in U$, written as $u=u_{I}+\widehat{w}$, because $U_{I}\perp_{a}\widehat
{W}$. Next, we note that the block $\mathcal{A}_{\Gamma\Gamma}:\widehat
{W}^{\prime}\rightarrow\widehat{W}$ is the Schur complement operator
corresponding to the form $s$. Finally, since the block $\mathcal{A}_{II}$\ is
used only in the preprocessing step, the preconditioned\ operator from
Algorithms~\ref{alg:bddc-elim-int} and~\ref{alg:bddc-elim-int-2} is simply
$M_{\Gamma\Gamma}\mathcal{A}_{\Gamma\Gamma}:r\in\widehat{W}^{\prime
}\rightarrow$ $u\in\widehat{W}.$

Let us now turn to Algorithm \ref{alg:original-bddc}. Let the residual $r\in
U$ be written as $r=r_{I}+r_{\Gamma}$, where $r_{I}\in U_{I}^{\prime}$ and
$r_{\Gamma}\in\widehat{W}^{\prime}$. Taking $r_{\Gamma}=0$, we get $r=r_{I}$,
and it follows that $r_{B}=u_{B}=v_{I}=0$, so $u=u_{I}$. On the other hand,
taking $r=r_{\Gamma}$ gives $u_{I}=0$, $r_{B}=r_{\Gamma}$, $v_{I}=Pu_{B}$\ and
finally $u=\left(  I-P\right)  E(w_{\Delta}+w_{\Pi})$, so $u\in\widehat{W}$.
This shows that the off-diagonal blocks of the preconditioner $M$ are zero,
and therefore it is block diagonal
\[
M=\left[
\begin{array}
[c]{cc}%
M_{II} & 0\\
0 & M_{\Gamma\Gamma}%
\end{array}
\right]  .
\]
Next, let us take $u=u_{I},$ and consider $r_{\Gamma}=0$. The
algorithm\ returns $r_{B}=u_{B}=v_{I}=0$, and finally $u=u_{I}$. This means
that $M_{II}\mathcal{A}_{II}u_{I}=u_{I},$ so $M_{II}=\mathcal{A}_{II}^{-1}$.
The operator $A:U\rightarrow U^{\prime}$, and the block preconditioned
operator $MA:r\in U^{\prime}\rightarrow u\in U$ from Algorithm
\ref{alg:original-bddc} can be written, respectively, as%
\[
A=\left[
\begin{array}
[c]{cc}%
\mathcal{A}_{II} & 0\\
0 & \mathcal{A}_{\Gamma\Gamma}%
\end{array}
\right]  ,\quad MA=\left[
\begin{array}
[c]{cc}%
I & 0\\
0 & M_{\Gamma\Gamma}\mathcal{A}_{\Gamma\Gamma}%
\end{array}
\right]  ,
\]
where the right lower block $M_{\Gamma\Gamma}\mathcal{A}_{\Gamma\Gamma}%
:r\in\widehat{W}^{\prime}\rightarrow$ $u\in\widehat{W}$ is exactly the same as
the preconditioned operator from Algorithms~\ref{alg:bddc-elim-int}
and~\ref{alg:bddc-elim-int-2}.
\end{proof}

The BDDC condition number estimate is well known from~\cite{Mandel-2003-CBD}.
Following Theorem \ref{thm:equiv} and Corollary~\ref{cor:multispace-bddc}, we
only need to estimate $\left\Vert \left(  I-P\right)  Ew\right\Vert _{a}$ on
$\widetilde{W}$.

\begin{theorem}
[\cite{Mandel-2003-CBD}]\label{thm:element-bound}The condition number of the
original BDDC algorithm satisfies ${\kappa\leq\omega}$, where%
\begin{equation}
\omega=\max\left\{  {\sup_{w\in\widetilde{W}}\frac{\left\Vert \left(
I-P\right)  Ew\right\Vert _{a}^{2}}{\left\Vert w\right\Vert _{a}^{2}}%
,1}\right\}  \text{ }\leq C\left(  1+\log\frac{H}{h}\right)  ^{2}{.}
\label{eq:element-bound}%
\end{equation}

\end{theorem}

\begin{remark}
In \cite{Mandel-2003-CBD}, the theorem was formulated by taking the supremum
over the space of discrete harmonic functions $(I-P)\widetilde{W}$. However,
the supremum remains the same by taking the larger space $\widetilde{W}%
\supset(I-P)\widetilde{W}$, since%
\[
{\frac{\left\Vert \left(  I-P\right)  Ew\right\Vert _{a}^{2}}{\left\Vert
w\right\Vert _{a}^{2}}\leq\frac{\left\Vert \left(  I-P\right)  E\left(
I-P\right)  w\right\Vert _{a}^{2}}{\left\Vert \left(  I-P\right)  w\right\Vert
_{a}^{2}}}%
\]
from $E\left(  I-P\right)  =E$, which follows from (\ref{eq:int-unchanged}),
and from $\left\Vert w\right\Vert _{a}\geq\left\Vert \left(  I-P\right)
w\right\Vert _{a}$, which follows from the $a$-orthogonality of the projection
$P$.
\end{remark}

Before proceeding into the Multilevel BDDC section, let us write concisely the
spaces and operators involved in the two-level preconditioner as%
\[
U_{I}{\textstyle%
\genfrac{}{}{0pt}{}{\genfrac{}{}{0pt}{}{P}{\leftarrow}}{\subset}%
}U{\textstyle%
\genfrac{}{}{0pt}{}{\genfrac{}{}{0pt}{}{E}{\leftarrow}}{\subset}%
}\widetilde{W}_{\Delta}\oplus\widetilde{W}_{\Pi}=\widetilde{W}\subset W.
\]
We are now ready to extend this decomposition into the multilevel case.

\section{Multilevel BDDC and an Abstract Bound}

\label{sec:multilevel-bddc}

In this section, we generalize the two-level BDDC preconditioner to multiple
levels, using the abstract Multispace BDDC framework from
Algorithm~\ref{alg:multispace-bddc}. The substructuring components from
Section~\ref{sec:one-level} will be denoted by an additional subscript~$_{1},$
as~$\Omega_{1}^{s},$ $s=1,\ldots N_{1}$, etc., and called level~$1$. The
level~$1$ coarse problem (\ref{eq:coarse}) will be called the level~$2$
problem. It has the same finite element structure as the original problem
(\ref{eq:problem}) on level~$1$, so we put $U_{2}=\widetilde{W}_{\Pi1}$.
Level~$1 $ substructures are level~$2$ elements and level $1$ coarse degrees
of freedom are level~$2$ degrees of freedom. Repeating this process
recursively, level~$i-1$ substructures become level~$i$ elements, and the
level~$i$ substructures are agglomerates of level~$i$ elements. Level~$i$
substructures are denoted by $\Omega_{i}^{s},$ $s=1,\ldots,N_{i},$ and they
are assumed to form a conforming triangulation with a characteristic
substructure size $H_{i}$. For convenience, we denote by $\Omega_{0}^{s}$ the
original finite elements and put $H_{0}=h$. The interface$~\Gamma_{i}$\ on
level$~i$ is defined as the union of all level$~i$ boundary nodes, i.e., nodes
shared by at least two level$~i$ substructures, and we note that $\Gamma
_{i}\subset\Gamma_{i-1}$. Level $i-1$ coarse degrees of freedom become level
$i$ degrees of freedom. The shape functions on level~$i$ are determined by
minimization of energy with respect to level~$i-1$ shape functions, subject to
the value of exactly one level $i$ degree of freedom being one and others
level $i$ degrees of freedom being zero. The minimization is done on each
level $i$ element (level $i-1$ substructure) separately, so the values of
level $i-1$ degrees of freedom are in general discontinuous between level
$i-1$ substructures, and only the values of level $i$\ degrees of freedom
between neighboring level~$i$\ elements coincide.

The development of the spaces on level $i$ now parallels the finite element
setting in Section \ref{sec:FE-setting}. Denote $U_{i}=\widetilde{W}_{\Pi
i-1}$. Let $W_{i}^{s}$ be the space of functions\ on the substructure
$\Omega_{i}^{s}$, such that all of their degrees of freedom on $\partial
\Omega_{i}^{s}\cap\partial\Omega$ are zero, and let%
\[
W_{i}=W_{i}^{1}\times\cdots\times W_{i}^{N_{i}}.
\]
Then $U_{i}\subset W_{i}$ is the subspace of all functions from $W$ that are
continuous across the interfaces $\Gamma_{i}$. Define $U_{Ii}\subset U_{i}$ as
the subspace of functions that are zero on$~\Gamma_{i}$, i.e., the functions
\textquotedblleft interior\textquotedblright\ to the level$~i$ substructures.
Denote by $P_{i}$ the energy orthogonal projection from $W_{i} $ onto $U_{Ii}%
$,%
\[
P_{i}:w_{i}\in W_{i}\longmapsto v_{Ii}\in U_{Ii}:a\left(  v_{Ii}%
,z_{Ii}\right)  =a\left(  w_{i},z_{Ii}\right)  ,\quad\forall z_{Ii}\in
U_{Ii}.
\]
Functions from $\left(  I-P_{i}\right)  W_{i}$, i.e., from the nullspace of
$P_{i},$ are called discrete harmonic on level $i$; these functions are
$a$-orthogonal to $U_{Ii}$ and energy minimal with respect to increments in
$U_{Ii}$. Denote by $\widehat{W}_{i}\subset U_{i}$ the subspace of discrete
harmonic functions on level$~i$, that is
\begin{equation}
\widehat{W}_{i}=\left(  I-P_{i}\right)  U_{i}. \label{eq:discrete-harm-ML}%
\end{equation}
In particular, $U_{Ii}\perp_{a}\widehat{W}_{i}$. Define $\widetilde{W}%
_{i}\subset W_{i}$ as the subspace of all functions such that the values of
any coarse degrees of freedom on level$~i$ have a common value for all
relevant level$~i$\ substructures and vanish on $\partial\Omega_{i}^{s}%
\cap\partial\Omega,$ and $\widetilde{W}_{\Delta i}\subset W_{i}$ as the
subspace of all functions such that their level $i$ coarse degrees of freedom
vanish. Define $\widetilde{W}_{\Pi i}$ as the subspace of all functions such
that their level $i$\ coarse degrees of freedom between adjacent substructures
coincide, and such that their energy is minimal. Clearly, functions in
$\widetilde{W}_{\Pi i}$ are uniquely determined by the values of their level
$i$\ coarse degrees of freedom, and
\begin{equation}
\widetilde{W}_{\Delta i}\perp_{a}\widetilde{W}_{\Pi i},\text{\quad}%
\widetilde{W}_{i}=\widetilde{W}_{\Delta i}\oplus\widetilde{W}_{\Pi i}.
\label{eq:tilde-dec-ML}%
\end{equation}
We assume that the level$~i$\ coarse degrees of freedom are zero on all
functions from $U_{Ii}$, that is,%
\begin{equation}
U_{Ii}\subset\widetilde{W}_{\Delta i}. \label{eq:coarse-int-ML}%
\end{equation}
In other words, level $i$ coarse degrees of freedom depend on the values on
level$~i$ substructure boundaries only. From (\ref{eq:tilde-dec-ML}) and
(\ref{eq:coarse-int-ML}), it follows that the functions in $\widetilde{W}_{\Pi
i}$ are discrete harmonic on level$~i$, that is%
\begin{equation}
\widetilde{W}_{\Pi i}=\left(  I-P_{i}\right)  \widetilde{W}_{\Pi i}.
\label{eq:coarse-is-discrete-harmonic-ML}%
\end{equation}
Let $E$ be a projection from $\widetilde{W}_{i}$ onto $U_{i}$, defined by
taking some weighted average on $\Gamma_{i}$%
\[
E_{i}:\widetilde{W}_{i}\rightarrow U_{i},\quad E_{i}U_{Ii}=U_{Ii},\quad
E_{i}^{2}=E_{i}.
\]
Since projection is the identity on its range,\ $E_{i}$ does not change the
level$~i$ interior degrees of freedom, in particular%
\begin{equation}
E_{i}U_{Ii}=U_{Ii}. \label{eq:int-unchanged-ML}%
\end{equation}

Finally, we introduce an interpolation $I_{i}:U_{i}\rightarrow\widetilde
{U}_{i}$ from level~$i$ degrees of freedom to functions in some classical
finite element space $\widetilde{U}_{i}$ with the same degrees of freedom as
$U_{i}$. The space $\widetilde{U}_{i}$ will be used for comparison purposes,
to invoke known inequalities for finite elements. A more detailed description
of the properties of $I_{i}$\ and the spaces $\widetilde{U}_{i}$\ is postponed
to the next section.

The hierarchy of spaces and operators is shown concisely in
Figure~\ref{fig:multi-bddc}. The Multilevel BDDC\ method is defined
recursively \cite{Dohrmann-2003-PSC,Mandel-2007-OMB} by solving the coarse
problem on level $i $ only approximately, by one application of the
preconditioner on level $i-1$. Eventually, at level, $L-1$, the coarse
problem, which is the level $L$ problem, is solved exactly. We need a more
formal description of the method here, which is provided by the following algorithm.

\begin{figure}[ptb]
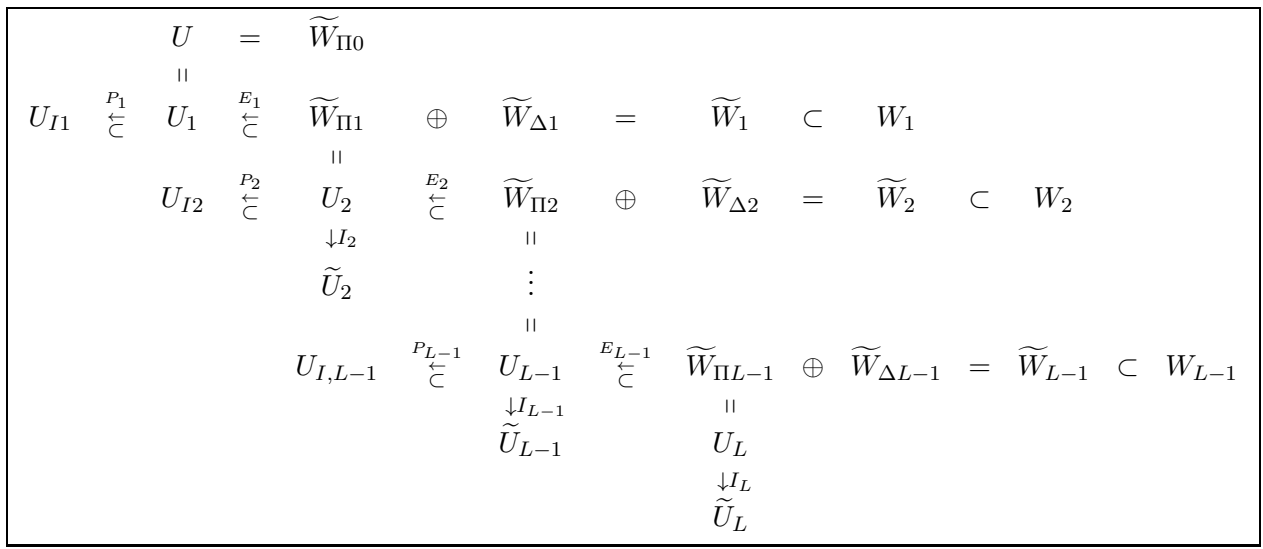

\medskip%
\[
\fbox{$%
\begin{array}
[c]{ccccccccccccccc}
&  & U & = & \widetilde{W}_{\Pi0} &  &  &  &  &  &  &  &  &  & \\
&  & \shortparallel &  &  &  &  &  &  &  &  &  &  &  & \\
U_{I1} & {\genfrac{}{}{0pt}{}{\genfrac{}{}{0pt}{}{P_{1}}{\leftarrow}}{\subset
}} & U_{1} & {\genfrac{}{}{0pt}{}{\genfrac{}{}{0pt}{}{E_{1}}{\leftarrow
}}{\subset}} & \widetilde{W}_{\Pi1} & \oplus & \widetilde{W}_{\Delta1} & = &
\widetilde{W}_{1} & \subset & W_{1} &  &  &  & \\
&  &  &  & \shortparallel &  &  &  &  &  &  &  &  &  & \\
&  & U_{I2} & {\genfrac{}{}{0pt}{}{\genfrac{}{}{0pt}{}{P_{2}}{\leftarrow
}}{\subset}} & U_{2} & {\genfrac{}{}{0pt}{}{\genfrac{}{}{0pt}{}{E_{2}%
}{\leftarrow}}{\subset}} & \widetilde{W}_{\Pi2} & \oplus & \widetilde
{W}_{\Delta2} & = & \widetilde{W}_{2} & \subset & W_{2} &  & \\
&  &  &  & {\scriptstyle\ \downarrow I_{2}} &  & \shortparallel &  &  &  &  &
&  &  & \\
&  &  &  & \widetilde{U}_{2} &  & \vdots &  &  &  &  &  &  &  & \\
&  &  &  &  &  & \shortparallel &  &  &  &  &  &  &  & \\
&  &  &  & U_{I,L-1} & {\genfrac{}{}{0pt}{}{\genfrac{}{}{0pt}{}{P_{L-1}%
}{\leftarrow}}{\subset}} & U_{L-1} &
{\genfrac{}{}{0pt}{}{\genfrac{}{}{0pt}{}{E_{L-1}}{\leftarrow}}{\subset}} &
\widetilde{W}_{\Pi L-1} & \oplus & \widetilde{W}_{\Delta L-1} & = &
\widetilde{W}_{L-1} & \subset & W_{L-1}\\
&  &  &  &  &  & {\scriptstyle\ \downarrow I_{L-1}} &  & \shortparallel &  &
&  &  &  & \\
&  &  &  &  &  & \widetilde{U}_{L-1} &  & U_{L} &  &  &  &  &  & \\
&  &  &  &  &  &  &  & {\scriptstyle\ \downarrow I_{L}} &  &  &  &  &  & \\
&  &  &  &  &  &  &  & \widetilde{U}_{L} &  &  &  &  &  &
\end{array}
$}%
\]
\medskip\caption{Space decompositions, embeddings and projections in the
Multilevel BDDC}%
\label{fig:multi-bddc}%
\end{figure}

\begin{algorithm}
[Multilevel BDDC]\label{alg:multilevel-bddc}Define the preconditioner
$r_{1}\in U_{1}^{\prime}\longmapsto u_{1}\in U_{1}$ as follows:

\textbf{for }$i=1,\ldots,L-1$\textbf{,}

\begin{description}
\item Compute interior pre-correction on level $i$,%
\begin{equation}
u_{Ii}\in U_{Ii}:a\left(  u_{Ii},z_{Ii}\right)  =\left\langle r_{i}%
,z_{Ii}\right\rangle ,\quad\forall z_{Ii}\in U_{Ii}. \label{eq:ML-uIi}%
\end{equation}

\item Get updated residual on level $i$,%
\begin{equation}
r_{Bi}\in U_{i},\quad\left\langle r_{Bi},v_{i}\right\rangle =\left\langle
r_{i},v_{i}\right\rangle -a\left(  u_{Ii},v_{i}\right)  ,\quad\forall v_{i}\in
U_{i}. \label{eq:ML-rBi}%
\end{equation}

\item Find the substructure correction on level $i$:
\begin{equation}
w_{\Delta i}\in W_{\Delta i}:a\left(  w_{\Delta i},z_{\Delta i}\right)
=\left\langle r_{Bi},E_{i}z_{\Delta i}\right\rangle ,\quad\forall z_{\Delta
i}\in W_{\Delta i}. \label{eq:ML-wDi}%
\end{equation}

\item Formulate the coarse problem on level $i$,
\begin{equation}
w_{\Pi i}\in W_{\Pi i}:a\left(  w_{\Pi i},z_{\Pi i}\right)  =\left\langle
r_{Bi},E_{i}z_{\Pi i}\right\rangle ,\quad\forall z_{\Pi i}\in W_{\Pi i},
\label{eq:ML-coarse}%
\end{equation}

\item If\ $\ i=L-1$, solve the coarse problem directly and set $u_{L}=w_{\Pi
L-1}$, \newline otherwise set up the right-hand side for level $i+1$,%
\begin{equation}
r_{i+1}\in\widetilde{W}_{\Pi i}^{\prime},\quad\left\langle r_{i+1}%
,z_{i+1}\right\rangle =\left\langle r_{Bi},E_{i}z_{i+1}\right\rangle
,\quad\forall z_{i+1}\in\widetilde{W}_{\Pi i}=U_{i+1}, \label{eq:ML-ri+1}%
\end{equation}

\end{description}

\textbf{end.}

\textbf{for }$i=L-1,\ldots,1\mathbf{,}$\textbf{\ }

\begin{description}
\item Average the approximate corrections on substructure interfaces on level
$i$,%
\begin{equation}
u_{Bi}=E_{i}\left(  w_{\Delta i}+u_{i+1}\right)  . \label{eq:ML-uBi-1}%
\end{equation}

\item Compute the interior post-correction on level $i$,%
\begin{equation}
v_{Ii}\in U_{Ii}:a\left(  v_{Ii},z_{Ii}\right)  =a\left(  u_{Bi}%
,z_{Ii}\right)  ,\quad\forall z_{Ii}\in U_{Ii}. \label{eq:ML-vIi}%
\end{equation}

\item Apply the combined corrections,
\begin{equation}
u_{i}=u_{Ii}+u_{Bi}-v_{Ii}. \label{eq:ML-ui}%
\end{equation}

\end{description}

\textbf{end.}
\end{algorithm}

We can now show that the Multilevel BDDC can be cast as the Multispace BDDC on
energy orthogonal spaces, using the hierarchy of spaces from
Figure~\ref{fig:multi-bddc}.

\begin{lemma}
\label{lem:bddc-ML-spaces}The Multilevel BDDC\ preconditioner in
Algorithm~\ref{alg:multilevel-bddc} is the abstract Multispace BDDC
preconditioner from Algorithm~\ref{alg:multispace-bddc} with $M=2L-1,$ and the
spaces and operators%
\begin{align}
X  &  =U_{1},\quad V_{1}=U_{I1},\quad V_{2}=(I-P_{1})\widetilde{W}_{\Delta
1},\quad V_{3}=U_{I2},\nonumber\\
V_{4}  &  =(I-P_{2})\widetilde{W}_{\Delta2},\quad V_{5}=U_{I3},\quad
\ldots\label{eq:ML-spaces}\\
V_{2L-4}  &  =(I-P_{L-2})\widetilde{W}_{\Delta L-2},\quad V_{2L-3}%
=U_{IL-1},\nonumber\\
V_{2L-2}  &  =(I-P_{L-1})\widetilde{W}_{\Delta L-1},\quad V_{2L-1}%
=\widetilde{W}_{\Pi L-1},\nonumber
\end{align}%
\begin{align}
Q_{1}  &  =I,\quad Q_{2}=Q_{3}=\left(  I-P_{1}\right)  E_{1},\quad\nonumber\\
Q_{4}  &  =Q_{5}=\left(  I-P_{1}\right)  E_{1}\left(  I-P_{2}\right)
E_{2},\quad\ldots\label{eq:ML-operators}\\
Q_{2L-4}  &  =Q_{2L-3}=\left(  I-P_{1}\right)  E_{1}\,\cdots\,\left(
I-P_{L-2}\right)  E_{L-2},\nonumber\\
Q_{2L-2}  &  =Q_{2L-1}=\left(  I-P_{1}\right)  E_{1}\,\cdots\,\left(
I-P_{L-1}\right)  E_{L-1},\nonumber
\end{align}
and the assumptions of Corollary \ref{cor:multispace-bddc} are satisfied.
\end{lemma}

\begin{proof}
Let $r_{1}\in U_{1}^{\prime}$. Define the vectors $v_{k},$\ $k=1,\ldots
,2L-1$\ by (\ref{def:abs-mult-BDDC}) with the spaces and operators given by
(\ref{eq:ML-spaces})-(\ref{eq:ML-operators}), and let $u_{Ii}$, $r_{Bi}$,
$w_{\Delta i}$, $w_{\Pi i}$, $r_{i+1}$, $u_{Bi}$, $v_{Ii}$, and $u_{i}$ be the
quantities in Algorithm~\ref{alg:multilevel-bddc}, defined by (\ref{eq:ML-uIi}%
)-(\ref{eq:ML-ui}).

First, with $V_{1}=U_{I1}$, the definition of $v_{1}$\ in
(\ref{def:abs-mult-BDDC}) is (\ref{eq:ML-uIi}) with $i=1$ and $u_{I1}=v_{1}$.
We show that in general, for level $i=1,\ldots,L-1$, and space $k=2i-1,$\ we
get (\ref{def:abs-mult-BDDC}) with $V_{k}=U_{Ii}$, \ so that $v_{k}=u_{Ii}$
and in particular $v_{2L-3}=u_{IL-1}$. So, let $z_{Ii}\in U_{Ii}$,
$i=2,\ldots,L-1,$ be arbitrary. From (\ref{eq:ML-uIi}) using (\ref{eq:ML-ri+1}%
) and (\ref{eq:ML-rBi}),
\begin{align}
a(u_{Ii},z_{Ii})  &  =\left\langle r_{i},z_{Ii}\right\rangle =\left\langle
r_{Bi-1},E_{i-1}z_{Ii}\right\rangle =\label{eq:PF-uIi}\\
&  =\left\langle r_{i-1},E_{i-1}z_{Ii}\right\rangle -a\left(  u_{Ii-1}%
,E_{i-1}z_{Ii}\right)  .\nonumber
\end{align}
Since from (\ref{eq:ML-uIi}) using the fact that $P_{i-1}E_{i-1}z_{Ii}\in
U_{Ii-1}$ it follows that%
\[
\left\langle r_{i-1},P_{i-1}E_{i-1}z_{Ii}\right\rangle -a\left(
u_{Ii-1},P_{i-1}E_{i-1}z_{Ii}\right)  =0,
\]
we get from (\ref{eq:PF-uIi}),%
\[
a(u_{Ii},z_{Ii})=\left\langle r_{i-1},\left(  I-P_{i-1}\right)  E_{i-1}%
z_{Ii}\right\rangle -a\left(  u_{Ii-1},(I-P_{i-1})E_{i-1}z_{Ii}\right)  ,
\]
and because $a\left(  u_{Ii-1},(I-P_{i-1})E_{i-1}z_{Ii}\right)  =0$\ by
orthogonality, we get%
\[
a(u_{Ii},z_{Ii})=\left\langle r_{i-1},\left(  I-P_{i-1}\right)  E_{i-1}%
z_{Ii}\right\rangle .
\]
Repeating this process recursively using (\ref{eq:PF-uIi}), we finally get%
\begin{align*}
a(u_{Ii},z_{Ii})  &  =\left\langle r_{i-1},\left(  I-P_{i-1}\right)
E_{i-1}z_{Ii}\right\rangle =\quad...\\
&  =\left\langle r_{1},\left(  I-P_{1}\right)  E_{1}\,\cdots\,\left(
I-P_{i-1}\right)  E_{i-1}z_{Ii}\right\rangle .
\end{align*}

Next, consider $w_{\Delta i}\in\widetilde{W}_{\Delta i}$\ defined by
(\ref{eq:ML-wDi}). We show that for $i=1,\ldots,L-1$, and $k=2i$,\ we get
(\ref{def:abs-mult-BDDC})\ with $V_{k}=\widetilde{W}_{\Delta i}$,\ so that
$v_{k}=w_{\Delta i}$ and in particular $v_{2L-2}=w_{\Delta L-1}$. So, let
$z_{\Delta i}\in\widetilde{W}_{\Delta i}$ be arbitrary.\ From (\ref{eq:ML-wDi}%
)\ using (\ref{eq:ML-rBi}),
\begin{equation}
a\left(  w_{\Delta i},z_{\Delta i}\right)  =\left\langle r_{Bi},E_{i}z_{\Delta
i}\right\rangle =\left\langle r_{i},E_{i}z_{\Delta i}\right\rangle -a\left(
u_{Ii},E_{i}z_{\Delta i}\right)  . \label{eq:PF-wDelta-i}%
\end{equation}
From the definition of $u_{Ii}$ by (\ref{eq:ML-uIi}) and since $P_{i}%
E_{i}z_{\Delta i}\in U_{Ii}$ it follows that
\[
\left\langle r_{i},P_{i}E_{i}z_{\Delta i}\right\rangle -a\left(  u_{Ii}%
,P_{i}E_{i}z_{\Delta i}\right)  =0,
\]
so (\ref{eq:PF-wDelta-i}) gives%
\[
a\left(  w_{\Delta i},z_{\Delta i}\right)  =\left\langle r_{i},\left(
I-P_{i}\right)  E_{i}z_{\Delta i}\right\rangle -a\left(  u_{Ii},\left(
I-P_{i}\right)  E_{i}z_{\Delta i}\right)  .
\]
Next, because $a\left(  u_{Ii},\left(  I-P_{i}\right)  E_{i}z_{\Delta
i}\right)  =0$ by orthogonality, and using (\ref{eq:ML-ri+1}),
\[
a\left(  w_{\Delta i},z_{\Delta i}\right)  =\left\langle r_{i},\left(
I-P_{i}\right)  E_{i}z_{\Delta i}\right\rangle =\left\langle r_{Bi-1}%
,E_{i-1}\left(  I-P_{i}\right)  E_{i}z_{\Delta i}\right\rangle .
\]
Repeating this process recursively, we finally get%
\begin{align}
a\left(  w_{\Delta i},z_{\Delta i}\right)   &  =\left\langle r_{i},\left(
I-P_{i}\right)  E_{i}z_{\Delta i}\right\rangle =\quad\ldots\nonumber\\
&  =\left\langle r_{1},\left(  I-P_{1}\right)  E_{1}\,\cdots\,\left(
I-P_{i}\right)  E_{i}z_{\Delta i}\right\rangle .\nonumber
\end{align}
To verify (\ref{def:abs-mult-BDDC}), it remains to show that $P_{i}w_{\Delta
i}=0;$ then $w_{\Delta i}\in(I-P_{i})\widetilde{W}_{\Delta i}=V_{k}$. Since
$P_{i}$ is an $a$-orthogonal projection, it holds that%
\[
a\left(  P_{i}w_{\Delta i},P_{i}w_{\Delta i}\right)  =a\left(  w_{\Delta
i},P_{i}w_{\Delta i}\right)  =\left\langle r_{Bi},E_{i}P_{i}w_{\Delta
i}\right\rangle =0,
\]
where we have used $E_{i}U_{Ii}\subset U_{Ii}$ following the assumption
(\ref{eq:int-unchanged-ML}) and the equality%
\[
\left\langle r_{Bi},z_{Ii}\right\rangle =\left\langle r_{i},z_{Ii}%
\right\rangle -a\left(  u_{Ii},z_{Ii}\right)  =0
\]
for any $z_{Ii}\in U_{Ii}$, which follows from (\ref{eq:ML-uIi}) and
(\ref{eq:ML-rBi}).

In exactly the same way, we get that if $w_{\Pi L-1}\in\widetilde{W}_{\Pi
L-1}$\ is defined by (\ref{eq:ML-coarse}), then $v_{2L-1}=w_{\Pi L-1}%
$\ satisfies (\ref{def:abs-mult-BDDC}) with $k=2L-1$.

Finally, from (\ref{eq:ML-uBi-1})-(\ref{eq:ML-ui}) for any $i=L-2,\ldots,1$,
we get
\begin{align*}
u_{i}  &  =u_{Ii}+u_{Bi}-v_{Ii}\\
&  =u_{Ii}+\left(  I-P_{i}\right)  E_{i}\left(  w_{\Delta i}+u_{i+1}\right) \\
&  =u_{Ii}+\left(  I-P_{i}\right)  E_{i}\left[  w_{\Delta i}+u_{Ii+1}+\left(
I-P_{i+1}\right)  E_{i+1}\left(  w_{\Delta i+1}+u_{i+2}\right)  \right] \\
&  =u_{Ii}+\\
&  +\left(  I-P_{i}\right)  E_{i}\left[  w_{\Delta i}+\ldots+\left(
I-P_{L-1}\right)  E_{L-1}\left(  w_{\Delta L-1}+u_{\Pi L-1}\right)  \right]  ,
\end{align*}

and, in particular for $u_{1}$,
\begin{align*}
u_{1}  &  =u_{I1}+\\
&  +\left(  I-P_{1}\right)  E_{1}\left[  w_{\Delta1}+\ldots+\left(
I-P_{L-1}\right)  E_{L-1}\left(  w_{\Delta L-1}+u_{\Pi L-1}\right)  \right] \\
&  =Q_{1}v_{1}+Q_{2}v_{2}+\ldots+Q_{2L-2}v_{2L-2}+Q_{2L-1}v_{2L-1}.
\end{align*}

It remains to verify the assumptions of Corollary~\ref{cor:multispace-bddc}.

The spaces $\widetilde{W}_{\Pi i}$ and $\widetilde{W}_{\Delta i}$, for all
$i=1,\ldots,L-1$, are $a$-orthogonal by (\ref{eq:tilde-dec-ML}) and from
(\ref{eq:coarse-int-ML}),
\[
\left(  I-P_{i}\right)  \widetilde{W}_{\Delta i}\subset\widetilde{W}_{\Delta
i},
\]
thus $\left(  I-P_{i}\right)  \widetilde{W}_{\Delta i}$ is $a$-orthogonal to
$\widetilde{W}_{\Pi i}$. Since$\ \widetilde{W}_{\Pi i}=U_{i+1}$ consists of
discrete harmonic functions on level$~i$ from
(\ref{eq:coarse-is-discrete-harmonic-ML}), and $U_{Ii+1}\subset U_{i+1}$, it
follows by induction\ that the spaces $V_{k}$, given by (\ref{eq:ML-spaces}),
are $a$-orthogonal.

We now show that the operators $Q_{k}$ defined by (\ref{eq:ML-operators}) are
projections. From our definitions, coarse degrees of freedom on substructuring
level\ $i$ (from which we construct the level\thinspace$i+1$ problem) depend
only on the values of degrees of freedom on the interface$~\Gamma_{i}$ and
$\Gamma_{j}\subset\Gamma_{i}$ for $j\geq i$. Then,%
\begin{equation}
(I-P_{j})E_{j}(I-P_{i})E_{i}(I-P_{j})E_{j}=(I-P_{i})E_{i}(I-P_{j})E_{j}.
\label{eq:DH-unchanged-ML}%
\end{equation}

Using (\ref{eq:DH-unchanged-ML}) and since $\left(  I-P_{1}\right)  E_{1}$ is
a projection by (\ref{eq:(I-P)E}), we get%
\begin{align*}
\left[  \left(  I-P_{1}\right)  E_{1}\cdots\left(  I-P_{i}\right)
E_{i}\right]  ^{2}  &  =\left(  I-P_{1}\right)  E_{1}\left(  I-P_{1}\right)
E_{1}\cdots\left(  I-P_{i}\right)  E_{i}\\
&  =\left(  I-P_{1}\right)  E_{1}\cdots\left(  I-P_{i}\right)  E_{i},
\end{align*}
so the operators $Q_{k}$ from (\ref{eq:ML-operators}) are projections.

It remains to prove the decomposition of unity (\ref{eq:dec-unity}). Let
$u_{i}\in U_{i}$, such that%
\begin{align}
u_{i}^{\prime}  &  =u_{Ii}+w_{\Delta i}+u_{i+1},\qquad\label{eq:pf-dec-u}\\
u_{Ii}  &  \in U_{Ii},\quad w_{\Delta i}\in\left(  I-P_{i}\right)
\widetilde{W}_{\Delta i},\quad u_{i+1}\in U_{i+1}%
\end{align}
and%
\begin{equation}
v_{i}=u_{Ii}+\left(  I-P_{i}\right)  E_{i}w_{\Delta i}+\left(  I-P_{i}\right)
E_{i}u_{i+1}. \label{eq:pf-dec-v}%
\end{equation}
From (\ref{eq:pf-dec-u}), $w_{\Delta i}+u_{i+1}\in U_{i}$ since $u_{i}\in
U_{i}$ and $u_{Ii}\in U_{Ii}\subset U_{i}$. Then $E_{i}\left[  w_{\Delta
i}+u_{i+1}\right]  =w_{\Delta i}+u_{i+1}$\ by (\ref{eq:int-unchanged-ML}), so%
\begin{align*}
v_{i}  &  =u_{Ii}+(I-P_{i})E_{i}\left[  w_{\Delta i}+u_{i+1}\right]
=u_{Ii}+(I-P_{i})\left[  w_{\Delta i}+u_{i+1}\right]  =\\
&  =u_{Ii}+w_{\Delta i}+u_{i+1}=u_{Ii}+w_{\Delta i}+u_{i+1}=u_{i}^{\prime},
\end{align*}
because $w_{\Delta i}$ and $u_{i+1}$ are discrete harmonic on level $i$. The
fact that $u_{i+1}$ in (\ref{eq:pf-dec-u})\ and (\ref{eq:pf-dec-v}) are the
same on arbitrary level $i$ can be proved in exactly the same way using
induction and putting $u_{i+1}$ in (\ref{eq:pf-dec-u}) as%
\begin{align*}
u_{i+1}  &  =u_{Ii+1}+\ldots+w_{\Delta L-1}+w_{\Pi L-1},\\
u_{Ii+1}  &  \in U_{Ii+1},\quad w_{\Delta L-1}\in\left(  I-P_{L-1}\right)
\widetilde{W}_{\Delta L-1},\quad w_{\Pi L-1}\in\widetilde{W}_{\Pi L-1},
\end{align*}
and in (\ref{eq:pf-dec-v}) as
\[
u_{i+1}=u_{Ii+1}+\ldots+\left(  I-P_{i+1}\right)  E_{i+1}\cdots\left(
I-P_{L-1}\right)  E_{L-1}\left(  w_{\Delta L-1}+w_{\Pi L-1}\right)  .\;
\]

\end{proof}

The following bound follows from writing of the Multilevel BDDC as Multispace
BDDC in Lemma~\ref{lem:bddc-ML-spaces}\ and the estimate for Multispace BDDC
in Corollary~\ref{cor:multispace-bddc}.

\begin{lemma}
\label{lem:bddc-ML-estimate-MLspaces}If for some $\omega\geq1$,%
\begin{align}
\left\Vert (I-P_{1})E_{1}w_{\Delta1}\right\Vert _{a}^{2}  &  \leq
\omega\left\Vert w_{\Delta1}\right\Vert _{a}^{2}\quad\forall w_{\Delta1}%
\in\left(  I-P_{1}\right)  \widetilde{W}_{\Delta1},\nonumber\\
\left\Vert (I-P_{1})E_{1}u_{I2}\right\Vert _{a}^{2}  &  \leq\omega\left\Vert
u_{I2}\right\Vert _{a}^{2}\quad\forall u_{I2}\in U_{I2},\nonumber\\
&  \ldots\label{eq:est-omega-MLspaces}\\
\left\Vert (I-P_{1})E_{1}\,\cdots\,(I-P_{L-1})E_{L-1}w_{\Pi L-1}\right\Vert
_{a}^{2}  &  \leq\omega\left\Vert w_{\Pi L-1}\right\Vert _{a}^{2}\quad\forall
w_{\Pi L-1}\in\widetilde{W}_{\Pi L-1},\nonumber
\end{align}
then the Multilevel BDDC preconditioner (Algorithm \ref{alg:multilevel-bddc})
satisfies $\kappa\leq\omega.$
\end{lemma}

\begin{proof}
Choose the spaces and operators as in (\ref{eq:ML-spaces}%
)-(\ref{eq:ML-operators}) so that $u_{I1}=v_{1}\in V_{1}=U_{I1}$, $w_{\Delta
1}=v_{2}\in V_{2}=\left(  I-P_{1}\right)  \widetilde{W}_{\Delta1}$, $\ldots$,
$w_{\Pi L-1}=v_{2L-1}\in V_{2L-1}=\widetilde{W}_{\Pi L-1}$. The bound now
follows from Corollary~\ref{cor:multispace-bddc}.
\end{proof}

\begin{lemma}
\label{lem:bddc-ML-estimate}If for some $\omega_{i}\geq1$,
\begin{equation}
\left\Vert (I-P_{i})E_{i}w_{i}\right\Vert _{a}^{2}\leq\omega_{i}\left\Vert
w_{i}\right\Vert _{a}^{2},\quad\forall w_{i}\in\widetilde{W}_{i},\quad
i=1,\ldots,L-1, \label{eq:est-omega-k}%
\end{equation}
then the Multilevel BDDC preconditioner (Algorithm~\ref{alg:multilevel-bddc})
satisfies $\kappa\leq%
{\textstyle\prod_{i=1}^{L-1}}
\omega_{i}.$
\end{lemma}

\begin{proof}
Note from Lemma~\ref{lem:bddc-ML-estimate-MLspaces} that $\left(
I-P_{1}\right)  \widetilde{W}_{\Delta1}\subset\widetilde{W}_{\Delta1}%
\subset\widetilde{W}_{1}$, $U_{I2}\subset\widetilde{W}_{\Pi1}\subset
\widetilde{W}_{1}$, and generally $\left(  I-P_{i}\right)  \widetilde
{W}_{\Delta i}\subset\widetilde{W}_{\Delta i}\subset\widetilde{W}_{i}$,
$U_{Ii+1}\subset\widetilde{W}_{\Pi i}\subset\widetilde{W}_{i}.$
\end{proof}

\section{Condition Number Bound for the Model Problem}

\label{sec:multilevel-condition}

Let $\left\vert w\right\vert _{a(\Omega_{i}^{s})}$ be the energy norm of a
function $w\in\widetilde{W}_{\Pi i},$\ $i=1,\ldots,L-1,$\ restricted to
subdomain $\Omega_{i}^{s},$ $s=1,\ldots N_{i}$, i.e., $\left\vert w\right\vert
_{a(\Omega_{i}^{s})}^{2}=\int_{\Omega_{i}^{s}}\nabla w\nabla w\,dx,$ and let
$\left\Vert w\right\Vert _{a}$\ be the norm obtained by piecewise integration
over each $\Omega_{i}^{s}$. To apply Lemma~\ref{lem:bddc-ML-estimate} to the
model problem presented in Section~\ref{sec:one-level}, we need to generalize
the estimate from Theorem~\ref{thm:element-bound} to coarse levels. To this
end, let $I_{i+1}:\widetilde{W}_{\Pi i}\rightarrow\widetilde{U}_{i+1}$ be an
interpolation from the level$~i$ coarse degrees of freedom (i.e., level $i+1$
degrees of freedom) to functions in another space $\widetilde{U}_{i+1}$ and
assume that, for all $i=1,\ldots,L-1,$ and $s=1,\ldots,N_{i},$ the
interpolation satisfies for all $w\in\widetilde{W}_{\Pi i}$ and for all
$\Omega_{i+1}^{s}$ the equivalence
\begin{equation}
c_{i,1}\left\vert I_{i+1}w\right\vert _{a(\Omega_{i+1}^{s})}^{2}\leq\left\vert
I_{i}w\right\vert _{a(\Omega_{i+1}^{s})}^{2}\leq c_{i,2}\left\vert
I_{i+1}w\right\vert _{a(\Omega_{i+1}^{s})}^{2}, \label{eq:Ii-equiv}%
\end{equation}
which implies by Lemma~\ref{lem:equivalence} also the equivalence%
\begin{equation}
c_{i,1}\left\vert I_{i+1}w\right\vert _{H^{1/2}(\partial\Omega_{i+1}^{s})}%
^{2}\leq\left\vert I_{i}w\right\vert _{H^{1/2}(\partial\Omega_{i+1}^{s})}%
^{2}\leq c_{i,2}\left\vert I_{i+1}w\right\vert _{H^{1/2}(\partial\Omega
_{i+1}^{s})}^{2}, \label{eq:Ii-equiv-seminorm}%
\end{equation}
with $c_{i,2}/c_{i,1}\leq\operatorname*{const}$ bounded independently of
$H_{0},\ldots,H_{i+1}$.

\begin{remark}
Since $I_{1}=I$, the two norms are the same on $\widetilde{W}_{\Pi
0}=\widetilde{U}_{1}=U_{1}.$
\end{remark}

For the three-level BDDC in two dimensions, the result of Tu~\cite[Lemma
4.2]{Tu-2007-TBT}, which is based on the lower bound estimates by Brenner and
Sung~\cite{Brenner-2000-LBN}, can be written in our settings\ for all
$w\in\widetilde{W}_{\Pi1}$ and for all $\Omega_{2}^{s}$\ as%
\begin{equation}
c_{1,1}\left\vert I_{2}w\right\vert _{a(\Omega_{2}^{s})}^{2}\leq\left\vert
w\right\vert _{a(\Omega_{2}^{s})}^{2}\leq c_{1,2}\left\vert I_{2}w\right\vert
_{a(\Omega_{2}^{s})}^{2}, \label{eq:one-equiv-2D}%
\end{equation}
where $I_{2}$ is a piecewise (bi)linear interpolation given by values at
corners of level~$1$ substructures, and $c_{1,2}/c_{1,1}\leq
\operatorname*{const}$ independently of $H/h$.

For the three-level BDDC in three dimensions, the result of Tu~\cite[Lemma
4.5]{Tu-2007-TBT3D}, which is based on the lower bound estimates by Brenner
and He~\cite{Brenner-2003-LBT}, can be written in our settings\ for all
$w\in\widetilde{W}_{\Pi1}$ and for all $\Omega_{2}^{s}$\ as
\begin{equation}
c_{1,1}\left\vert I_{2}w\right\vert _{H^{1/2}(\partial\Omega_{2}^{s})}^{2}%
\leq\left\vert w\right\vert _{H^{1/2}(\partial\Omega_{2}^{s})}^{2}\leq
c_{1,2}\left\vert I_{2}w\right\vert _{H^{1/2}(\partial\Omega_{2}^{s})}^{2},
\label{eq:one-equiv-3D}%
\end{equation}
where $I_{2}$ is an interpolation from the coarse degrees of freedom given by
the averages over substructure edges, and $c_{1,2}/c_{1,1}\leq
\operatorname*{const}$ independently of $H/h$.

We note that the level~$2$ substructures are called subregions
in~\cite{Tu-2007-TBT,Tu-2007-TBT3D}. Since $I_{1}=I$, with $i=1$ the
equivalence (\ref{eq:one-equiv-2D})\ corresponds to (\ref{eq:Ii-equiv}), and
(\ref{eq:one-equiv-3D}) to (\ref{eq:Ii-equiv-seminorm}).

The next Lemma establishes the equivalence of seminorms on a factor space from
the equivalence of norms on the original space. Let $V\subset U$ be finite
dimensional spaces and $\left\Vert \cdot\right\Vert _{A}\ $a norm on $U$ and
define
\begin{equation}
\left\vert u\right\vert _{U/V,A}=\min\limits_{v\in V}\left\Vert u-v\right\Vert
_{A}. \label{eq:factor-norm}%
\end{equation}

We will be using (\ref{eq:factor-norm}) for the norm on the space of discrete
harmonic functions $(I-P_{i})\widetilde{W}_{i}$ with $V$ as the space of
interior functions $U_{Ii}$, and also with $V$ as the space $\widetilde
{W}_{\Delta i}$. In particular, since $\widetilde{W}_{\Pi i}\subset
(I-P_{i})\widetilde{W}_{i}$,\ we have%
\begin{equation}
w\in\widetilde{W}_{\Pi i},\quad\left\Vert w\right\Vert _{a}=\min_{w_{\Delta
}\in\widetilde{W}_{\Delta i}}\left\Vert w-w_{\Delta}\right\Vert _{a}
\label{eq:i-factor}%
\end{equation}

\begin{lemma}
\label{lem:equivalence}Let $\left\Vert \cdot\right\Vert _{A}$, $\left\Vert
\cdot\right\Vert _{B}$ be norms on $U$, and
\begin{equation}
c_{1}\left\Vert u\right\Vert _{B}^{2}\leq\left\Vert u\right\Vert _{A}^{2}\leq
c_{2}\left\Vert u\right\Vert _{B}^{2},\quad\forall u\in U.
\label{eq:equiv-norms}%
\end{equation}
Then for any subspace $V\subset U$,%
\[
c_{1}\left\vert u\right\vert _{U/V,B}^{2}\leq\left\vert u\right\vert
_{U/V,A}^{2}\leq c_{2}\left\vert u\right\vert _{U/V,B}^{2},
\]
resp.,
\begin{equation}
c_{1}\min_{v\in V}\left\Vert u-v\right\Vert _{B}^{2}\leq\min_{v\in
V}\left\Vert u-v\right\Vert _{A}^{2}\leq c_{2}\min_{v\in V}\left\Vert
u-v\right\Vert _{B}^{2}. \label{eq:equiv-factor}%
\end{equation}

\end{lemma}

\begin{proof}
From the definition (\ref{eq:factor-norm}) of the norm on a factor space, we
get
\[
\left\vert u\right\vert _{U/V,A}=\min\limits_{v\in V}\left\Vert u-v\right\Vert
_{A}=\left\Vert u-v_{A}\right\Vert _{A}.
\]
for some $v_{A}$. Let $v_{B}$ be defined similarly. Then
\begin{align*}
\left\vert u\right\vert _{U/V,A}^{2}  &  =\min\limits_{v\in V}\left\Vert
u-v\right\Vert _{A}^{2}=\left\Vert u-v_{A}\right\Vert _{A}^{2}\leq\left\Vert
u-v_{B}\right\Vert _{A}^{2}\leq\\
&  \leq c_{2}\left\Vert u-v_{B}\right\Vert _{B}^{2}=c_{2}\min\limits_{v\in
V}\left\Vert u-v\right\Vert _{B}^{2}=c_{2}\left\vert u\right\vert _{U/V,B}%
^{2},
\end{align*}
which is the right hand side inequality in (\ref{eq:equiv-factor}). The left
hand side inequality follows by switching the notation for $\left\Vert
\cdot\right\Vert _{A}$ and $\left\Vert \cdot\right\Vert _{B}$.
\end{proof}

\begin{lemma}
\label{lem:W_i norm equiv}For all $i=0,\ldots,L-1$, and $s=1,\ldots,N_{i}$,
\begin{equation}
c_{i,1}\left\vert I_{i+1}w\right\vert _{a(\Omega_{i+1}^{s})}^{2}\leq\left\vert
w\right\vert _{a(\Omega_{i+1}^{s})}^{2}\leq c_{k,2}\left\vert I_{i+1}%
w\right\vert _{a(\Omega_{i+1}^{s})}^{2},\quad\forall w\in\widetilde{W}_{\Pi
i},\;\forall\,\Omega_{i+1}^{s}, \label{eq:ML-equiv}%
\end{equation}
with $c_{i,2}/c_{i,1}\leq C_{i}$, independently of $H_{0}$,\ldots, $H_{i+1}$.
\end{lemma}

\begin{proof}
The proof follows by induction. For $i=0$, (\ref{eq:ML-equiv}) holds because
$I_{1}=I$. Suppose that (\ref{eq:ML-equiv}) holds for some $i<L-2$ and let
$w\in\widetilde{W}_{\Pi i+1}$. From the definition of $\widetilde{W}_{\Pi
i+1}$ by energy minimization,%
\begin{equation}
\left\vert w\right\vert _{a(\Omega_{i+1}^{s})}=\min_{w_{\Delta}\in
\widetilde{W}_{\Delta i+1}}\left\vert w-w_{\Delta}\right\vert _{a(\Omega
_{i+1}^{s})}. \label{eq:def-wa}%
\end{equation}
From (\ref{eq:def-wa}), the induction assumption, and
Lemma~\ref{lem:equivalence} eq.~(\ref{eq:equiv-factor}), it follows that%
\begin{align}
\lefteqn{c_{i,1}\min_{w_{\Delta}\in\widetilde{W}_{\Delta i+1}}\left\vert
I_{i+1}w-I_{i+1}w_{\Delta}\right\vert _{a(\Omega_{i+2}^{s})}^{2}%
}\label{eq:Ii1w}\\
&  \qquad\leq\min_{w_{\Delta}\in\widetilde{W}_{\Delta i+1}}\left\vert
w-w_{\Delta}\right\vert _{a(\Omega_{i+2}^{s})}^{2}\leq c_{i,2}\min_{w_{\Delta
}\in\widetilde{W}_{\Delta i+1}}\left\vert I_{i+1}w-I_{i+1}w_{\Delta
}\right\vert _{a(\Omega_{i+2}^{s})}^{2}\nonumber
\end{align}
From the assumption (\ref{eq:Ii-equiv}), applied to the functions of the form
$I_{i+1}w$ on $\Omega_{i+2}^{s}$,
\begin{equation}
c_{1}\left\vert I_{i+2}w\right\vert _{a(\Omega_{i+2}^{s})}^{2}\leq
\min_{w_{\Delta}\in\widetilde{W}_{\Delta i+1}}\left\vert I_{i+1}%
w-I_{i+1}w_{\Delta}\right\vert _{a(\Omega_{i+2}^{s})}^{2}\leq c_{2}\left\vert
I_{i+2}w\right\vert _{a(\Omega_{i+2}^{s})}^{2} \label{eq:Tuk}%
\end{equation}
with $c_{2}/c_{1}$, bounded independently of $H_{0},\ldots,H_{i+1}$. Then
(\ref{eq:def-wa}), (\ref{eq:Ii1w})\ and (\ref{eq:Tuk}) imply
(\ref{eq:ML-equiv}) with $C_{i}=C_{i-1}c_{2}/c_{1}$.
\end{proof}

Next, we generalize the estimate from Theorem~\ref{thm:element-bound} to
coarse levels.

\begin{lemma}
\label{lem:W_i operator equiv}For all substructuring levels $i=1,\ldots,L-1$,
\begin{equation}
\left\Vert (I-P_{i})E_{i}w_{i}\right\Vert _{a}^{2}\leq C_{i}\left(
1+\log\frac{H_{i}}{H_{i-1}}\right)  ^{2}\left\Vert w_{i}\right\Vert _{a}%
^{2},\quad\forall w_{i}\in U_{i}. \label{eq:ML-bound}%
\end{equation}

\end{lemma}

\begin{proof}
From (\ref{eq:ML-equiv}), summation over substructures on level$~i$ gives%
\begin{equation}
c_{i,1}\left\Vert I_{i}w\right\Vert _{a}^{2}\leq\left\Vert w\right\Vert
_{a}^{2}\leq c_{i,2}\left\Vert I_{i}w\right\Vert _{a}^{2},\quad\forall w\in
U_{i}. \label{eq:equiv-i1}%
\end{equation}
Next, in our context, using the definition of $P_{i}$ and
(\ref{eq:equiv-factor}), we get
\[
\left\Vert I_{i}(I-P_{i})E_{i}w_{i}\right\Vert _{a}^{2}=\min_{u_{Ii}\in
U_{Ii}}\left\Vert I_{i}E_{i}w_{i}-I_{i}u_{Ii}\right\Vert _{a}^{2},
\]
so from (\ref{eq:element-bound}) for some $\overline{C}_{i}$ and all
$i=1,\ldots,L-1$,
\begin{equation}
\min_{u_{Ii}\in U_{Ii}}\left\Vert I_{i}E_{i}w_{i}-I_{i}u_{Ii}\right\Vert
_{a}^{2}\leq\overline{C}_{i}\left(  1+\log\frac{H_{i}}{H_{i-1}}\right)
^{2}\left\Vert I_{i}w_{i}\right\Vert _{a}^{2},\quad\forall w_{i}\in U_{i}.
\label{eq:log-linear}%
\end{equation}
Similarly, from (\ref{eq:equiv-i1})\ and (\ref{eq:log-linear}) it follows that%
\begin{align*}
\left\Vert (I-P_{i})E_{i}w_{i}\right\Vert _{a}^{2}  &  =\min_{u_{Ii}\in
U_{Ii}}\left\Vert E_{i}w_{i}-u_{Ii}\right\Vert _{a}^{2}\\
&  \leq c_{i,2}\min_{u_{Ii}\in U_{Ii}}\left\Vert I_{i}E_{i}w_{i}-I_{i}%
u_{Ii}\right\Vert _{a}^{2}\\
&  \leq c_{i,2}\overline{C}_{i}\left(  1+\log\frac{H_{i}}{H_{i-1}}\right)
^{2}\left\Vert I_{i}w_{i}\right\Vert _{a}^{2}\\
&  \leq\frac{c_{i,2}\overline{C}_{i}}{c_{i,1}}\left(  1+\log\frac{H_{i}%
}{H_{i-1}}\right)  ^{2}\left\Vert w_{i}\right\Vert _{a}^{2},
\end{align*}
which is (\ref{eq:ML-bound}) with%
\[
C_{i}=\frac{c_{i,2}\overline{C}_{i}}{c_{i,1}},
\]
and $c_{i,2}/c_{i,1}$ from Lemma~\ref{lem:W_i norm equiv}.
\end{proof}

\begin{theorem}
\label{thm:ML-bound}The Multilevel BDDC for the model problem and corner
coarse function in 2D and edge coarse functions in 3D satisfies the condition
number estimate%
\[
\kappa\leq%
{\textstyle\prod_{i=1}^{L-1}}
C_{i}\left(  1+\log\frac{H_{i}}{H_{i-1}}\right)  ^{2}.
\]

\end{theorem}

\begin{proof}
The proof follows from Lemmas~\ref{lem:bddc-ML-estimate} and
\ref{lem:W_i operator equiv}, with $\omega_{i}=C_{i}\left(  1+\log\frac{H_{i}%
}{H_{i-1}}\right)  ^{2}$.
\end{proof}

\begin{remark}
For $L=3$ in two and three dimensions we recover the estimates by
Tu~\cite{Tu-2007-TBT,Tu-2007-TBT3D}, respectively.
\end{remark}

\begin{remark}
\label{rem:nonmonotone}While for standard (two-level) BDDC it is immediate
that increasing the coarse space and thus decreasing the space $\widetilde{W}$
cannot increase the condition number bound, this is an open problem for the
multilevel method. In fact, the 3D numerical results in the next section
suggest that this may not be  the case.
\end{remark}

\begin{corollary}
In the case of uniform coarsening, i.e. with $H_{i}/H_{i-1}=H/h$ and the same
geometry of decomposition on all levels $i=1,\ldots L-1,$ we get
\begin{equation}
\kappa\leq C^{L-1}\left(  1+\log H/h\right)  ^{2\left(  L-1\right)  }.
\label{eq:all-same}%
\end{equation}

\end{corollary}

\section{Numerical Examples}

\label{sec:numerical-examples}

\begin{table}[t]
\caption{Two dimensional (2D) results. The letters C and E designate the use
of corners and edges, respectively, in the coarse space.}
\begin{center}
{\footnotesize
\begin{tabular}
[c]{|c|c|c|c|c|c|c|}\hline
$L$ & \multicolumn{2}{c|}{C} & \multicolumn{2}{c|}{C+E} & $n$ & $n_{\Gamma}%
$\\\hline
& iter & cond & iter & cond &  & \\\hline
\multicolumn{7}{|c|}{$H_{i}/H_{i-1}=3$ at all levels.}\\\hline
2 & 8 & 1.92 & 5 & 1.08 & 144 & 80\\
3 & 13 & 3.10 & 7 & 1.34 & 1296 & 720\\
4 & 17 & 5.31 & 9 & 1.60 & 11,664 & 6480\\
5 & 23 & 9.22 & 10 & 1.85 & 104,976 & 58,320\\
6 & 31 & 16.07 & 11 & 2.12 & 944,748 & 524,880\\
7 & 42 & 28.02 & 13 & 2.45 & 8,503,056 & 4,723,920\\\hline
\multicolumn{7}{|c|}{$H_{i}/H_{i-1}=4$ at all levels.}\\\hline
2 & 9 & 2.20 & 6 & 1.14 & 256 & 112\\
3 & 15 & 4.02 & 8 & 1.51 & 4096 & 1792\\
4 & 21 & 7.77 & 10 & 1.88 & 65,536 & 28,672\\
5 & 30 & 15.2 & 12 & 2.24 & 1,048,576 & 458,752\\
6 & 42 & 29.7 & 13 & 2.64 & 16,777,216 & 7,340,032\\\hline
\multicolumn{7}{|c|}{$H_{i}/H_{i-1}=8$ at all levels.}\\\hline
2 & 10 & 2.99 & 7 & 1.33 & 1024 & 240\\
3 & 19 & 7.30 & 11 & 2.03 & 65,536 & 15,360\\
4 & 31 & 18.6 & 13 & 2.72 & 4,194,304 & 983,040\\
5 & 50 & 47.38 & 15 & 3.40 & 268,435,456 & 62,914,560\\\hline
\multicolumn{7}{|c|}{$H_{i}/H_{i-1}=12$ at all levels.}\\\hline
2 & 11 & 3.52 & 8 & 1.46 & 2304 & 368\\
3 & 21 & 10.12 & 12 & 2.39 & 331,776 & 52,992\\
4 & 39 & 29.93 & 15 & 3.32 & 47,775,744 & 7,630,848\\\hline
\multicolumn{7}{|c|}{$H_{i}/H_{i-1}=16$ at all levels.}\\\hline
2 & 11 & 3.94 & 8 & 1.56 & 4096 & 496\\
3 & 23 & 12.62 & 13 & 2.67 & 1,048,576 & 126,976\\
4 & 43 & 41.43 & 16 & 3.78 & 268,435,456 & 32,505,856\\\hline
\end{tabular}
}
\label{tab1}
\end{center}
\end{table}

\begin{table}[t]
\caption{Three dimensional (3D) results. The letters C, E, and F designate the
use of corners, edges, and faces, respectively, in the coarse space.}%
\label{tab2}
\begin{center}
{\footnotesize
\begin{tabular}
[c]{|c|c|c|c|c|c|c|c|c|}\hline
$L$ & \multicolumn{2}{c|}{E} & \multicolumn{2}{c|}{C+E} &
\multicolumn{2}{c|}{C+E+F} & $n$ & $n_{\Gamma}$\\\hline
& iter & cond & iter & cond & iter & cond &  & \\\hline
\multicolumn{7}{|c|}{$H_{i}/H_{i-1}=3$ at all levels.} &
\multicolumn{2}{|c|}{}\\\hline
2 & 10 & 1.85 & 8 & 1.47 & 5 & 1.08 & 1728 & 1216\\
3 & 14 & 3.02 & 12 & 2.34 & 8 & 1.50 & 46,656 & 32,832\\
4 & 18 & 4.74 & 18 & 5.21 & 11 & 2.20 & 1,259,712 & 886,464\\
5 & 23 & 7.40 & 26 & 14.0 & 16 & 3.98 & 34,012,224 & 23,934,528\\\hline
\multicolumn{7}{|c|}{$H_{i}/H_{i-1}=4$ at all levels.} &
\multicolumn{2}{|c|}{}\\\hline
2 & 10 & 1.94 & 9 & 1.66 & 6 & 1.16 & 4096 & 2368\\
3 & 15 & 3.51 & 14 & 3.24 & 10 & 1.93 & 262,144 & 151,552\\
4 & 20 & 6.09 & 22 & 9.95 & 14 & 3.05 & 16,777,216 & 9,699,328\\\hline
\multicolumn{7}{|c|}{$H_{i}/H_{i-1}=8$ at all levels.} &
\multicolumn{2}{|c|}{}\\\hline
2 & 12 & 2.37 & 11 & 2.24 & 8 & 1.50 & 32,768 & 10,816\\
3 & 19 & 5.48 & 20 & 7.59 & 14 & 3.32 & 16,777,216 & 5,537,792\\\hline
\multicolumn{7}{|c|}{$H_{i}/H_{i-1}=10$ at all levels.} &
\multicolumn{2}{|c|}{}\\\hline
2 & 12 & 2.56 & 12 & 2.47 & 9 & 1.69 & 64,000 & 17,344\\
3 & 20 & 6.39 & 22 & 10.1 & 16 & 3.85 & 64,000,000 & 17,344,000\\\hline
\end{tabular}
}
\end{center}
\end{table}

\begin{figure}[ptb]
\centering \scalebox{0.70} { \includegraphics{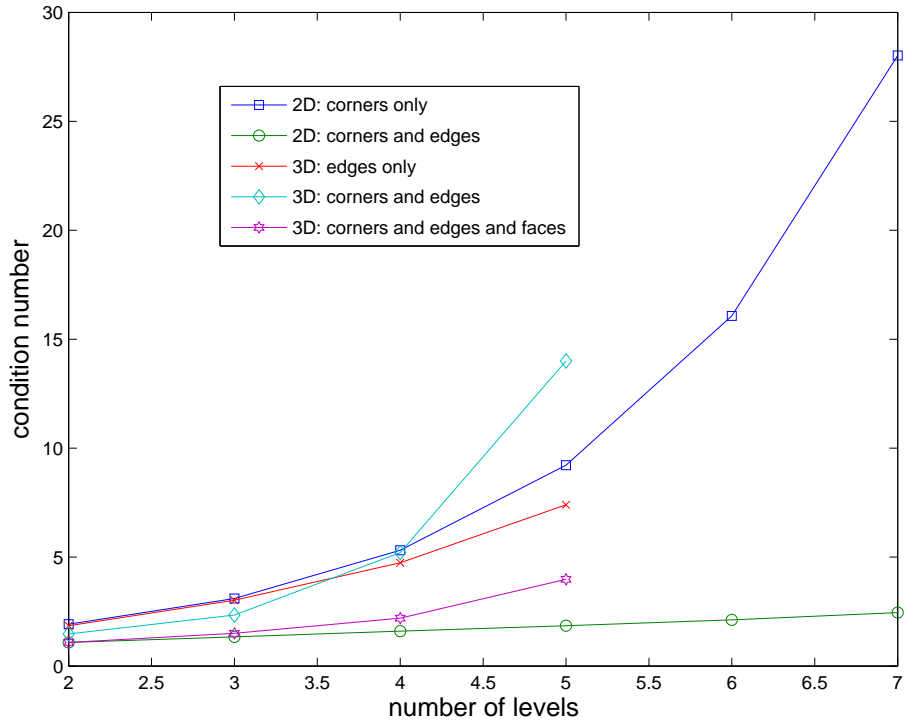} }\caption{Plot of data
in Tables~\ref{tab1} and \ref{tab2} for $H_{i}/H_{i-1}=3$ at all levels.}%
\label{ne:fig1}%
\end{figure}

\begin{figure}[ptb]
\centering \scalebox{0.70} { \includegraphics{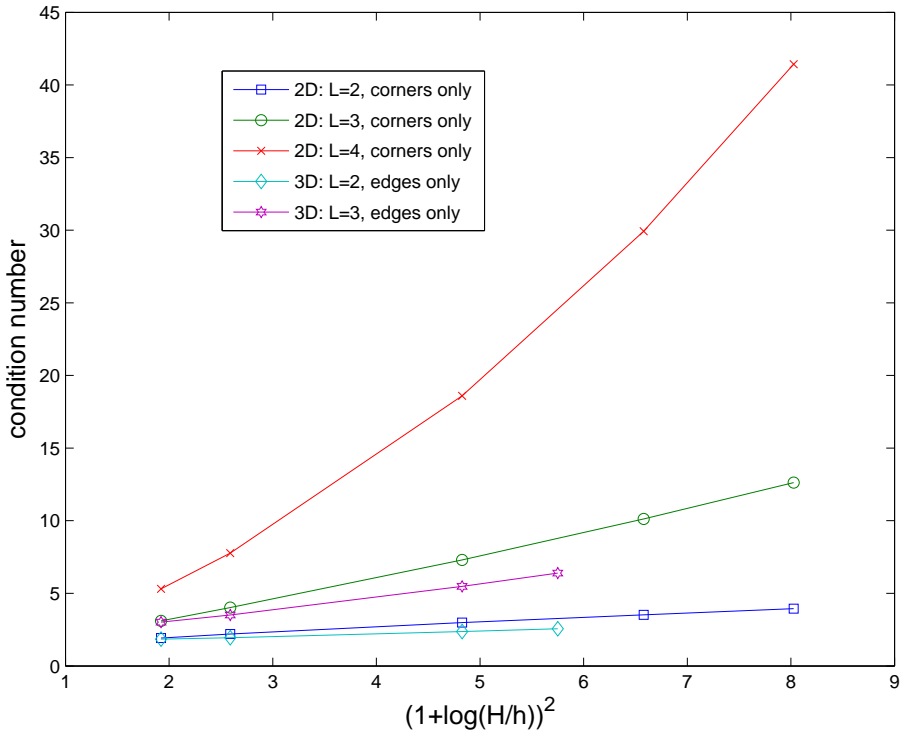} }\caption{Plot of data
in Tables~\ref{tab1} and \ref{tab2} for uniform coarsening.}%
\label{ne:fig2}%
\end{figure}

Numerical examples are presented in this section for the Poisson equation in
two and three dimensions. The problem domain in 2D (3D) is the unit square
(cube), and standard bilinear (trilinear) finite elements are used for the
discretization. The substructures at each level are squares or cubes, and
periodic essential boundary conditions are applied to the boundary of the
domain. This choice of boundary conditions allows us to solve very large
problems on a single processor since all substructure matrices are identical
for a given level.

The preconditioned conjugate gradient algorithm is used to solve the
associated linear systems to a relative residual tolerance of $10^{-8}$ for
random right-hand-sides with zero mean value. The zero mean condition is
required since, for periodic boundary conditions, the null space of the
coefficient matrix is the unit vector. The coarse problem always has $4^{2}$
($4^{3}$) subdomains at the coarsest level for 2D (3D) problems.

The number of levels ($L$), the number of iterations (iter), and condition
number estimates (cond) obtained from the conjugate gradient iterations are
reported in Tables~\ref{tab1} and \ref{tab2}. The letters C, E, and F
designate the use of corners, edges, or faces in the coarse space. For
example, C+E means that both corners and edges are used in the coarse space.
For 2D and 3D problems, the theory is applicable to coarse spaces C and E,
respectively. Also shown in the tables are the total number of unknowns ($n$)
and the number of unknowns ($n_{\Gamma}$) on subdomain boundaries at the
finest level.

The results in Tables~\ref{tab1} and \ref{tab2} are displayed in
Figure~\ref{ne:fig1} for a fixed value of $H_{i}/H_{i-1}=3$. In two dimensions
we observe very different behavior depending on the particular form of the
coarse space. If only corners are used in 2D, then there is very rapid growth
of the condition number with increasing numbers of levels as predicted by the
theory. In contrast, if both corners and edges are used in the 2D coarse
space, then the condition number appears to vary linearly with $L$ for the the
number of levels considered. Our explanation is that a bound similar to
Theorem~\ref{thm:ML-bound} still applies to the favorable 2D case, though
possibly with (much) smaller constants, so the exponential growth of the
condition number is no longer apparent. The results in Tables~\ref{tab1} and
\ref{tab2} are also displayed in Figure~\ref{ne:fig2} for fixed numbers of
levels. The observed growth of condition numbers for the case of uniform
coarsening is consistent with the estimate in (\ref{eq:all-same}).

Similar trends are present in 3D, but the beneficial effects of using more
enriched coarse spaces are much less pronounced. Interestingly, when comparing
the use of edges only (E) with corners and edges (C+E) in the coarse space,
the latter does not always lead to smaller numbers of iterations or condition
numbers for more than two levels. The fully enriched coarse space (C+E+F),
however, does give the best results in terms of iterations and condition
numbers. It should be noted that the present 3D theory in Theorem
\ref{thm:ML-bound} covers only the use of the edges only, and the present
theory does not guarantee that the condition number (or even its
bound)\ decrease with increasing the coarse space (Remark \ref{rem:nonmonotone}%
).

In summary, the numerical examples suggest that better performance, especially
in 2D, can be obtained when using a fully enriched coarse space. Doing so does
not incur a large computational expense since there is never the need to solve
a large coarse problem exactly with the multilevel approach. Finally, we note
that a large number of levels is not required to solve very large problems.
For example, the number of unknowns in 3D for a 4-level method with a
coarsening ratio of $H_{i}/H_{i-1}=10$ at all levels is $(10^{4})^{3} =
10^{12}$.

\bibliographystyle{siam}
\bibliography{../../bibliography/bddc}
\end{document}